\documentclass[a4paper,11pt,reqno]{amsart}
\usepackage{amssymb}        \usepackage{amsmath}
\usepackage{latexsym}       \usepackage{amsthm}
\usepackage{amsgen}         \usepackage{amscd}
\usepackage{mathtools}
\usepackage[all]{xy}       \usepackage{amsopn}
\usepackage{amstext}        \usepackage{enumerate}
\usepackage{mathrsfs}
\usepackage{leftidx}

\newtheorem{theorem}{Theorem}[section]
\newtheorem{proposition}[theorem]{Proposition}
\newtheorem{lemma}[theorem]{Lemma}
\newtheorem{corollary}[theorem]{Corollary}

\theoremstyle{definition}
\newtheorem{definition}{Definition}[section]
\newtheorem*{remark}{Remark}

\newtheorem*{notation}{Notation}
\newtheorem*{acknowledgment}{Acknowledgment}

\numberwithin{equation}{section}

\newcommand{\Z}{\mathbb{Z}}
\newcommand{\Q}{\mathbb{Q}}

\newcommand{\C}{\mathbb{C}}
\newcommand{\A}{\mathbb{A}}

\newcommand{\G}{\varGamma}
\newcommand{\g}{\gamma}

\newcommand{\la}{\lambda}
\newcommand{\ve}{\varepsilon}

\DeclareMathOperator{\Tr}{Tr}
\DeclareMathOperator{\re}{Re}

\DeclareMathOperator{\Sym}{Sym}

\DeclareMathOperator{\sgn}{sgn}
\DeclareMathOperator{\diag}{diag}
\DeclareMathOperator{\rank}{rank}
\DeclareMathOperator{\ord}{ord}

\newcommand{\Sp}{\mathit{Sp}}
\newcommand{\GL}{\mathit{GL}}
\newcommand{\SL}{\mathit{SL}}

\newcommand{\sikakubox}{\leavevmode
  \hbox to.77778em{%
  \hfil\vrule
  \vbox to.775em{\hrule width.25em\vfil\hrule}%
  \vrule\hfil}}

\newenvironment{smatrix}{\left( \begin{smallmatrix}}{\end{smallmatrix} \right)}

\newcommand{\ha}{\mathbb{H}}

\newcommand{\epi}{\mathbf{e}}
\newcommand{\fd}{\mathfrak{d}}

\newcommand{\bfI}{\mathbf{1}}

\title[Fourier coefficients of Siegel Eisenstein series of prime level]{Recursion formulas for the Fourier coefficients of Siegel Eisenstein series of an odd prime level}
\author{Keiichi Gunji}
\date{}

\email{keiichi.gunji@it-chiba.ac.jp}
\address{Department of Mathematics, Chiba Institute of Technology,
 2-1-1 Shibazono, Narashino, Chiba, 275-0023, Japan}

\begin{document}
\maketitle

\begin{abstract} In this paper we treat the Fourier coefficients of Siegel Eisenstein series of level $p$ with trivial or quadratic character, for an odd prime $p$. The Euler $p$-factor of the Fourier coefficient is called the ramified Siegel series. First we show that the ramified Siegel series attached to each cusp can be decomposed to $U(p)$-eigenfunctions explicitly, next we give recursion formulas of such $U(p)$-characteristic ramified Siegel series.
\end{abstract}

\section{Introduction} \label{Sec_introduction}

Although Siegel Eisenstein series is one of the most basic objects in the theory of Siegel modular forms, its Fourier coefficients are mysterious and quite difficult to calculate explicitly. The Fourier coefficients have Euler product expressions and each local factor is called the \emph{Siegel series}.
 In 1999 Katsurada gave an explicit formula of the Siegel series, hence the formula of the Fourier coefficients, for full modular case in \cite{Kat}. His method is consist of two pillars, one is the inductive relation and the other is the functional equation of the Siegel series. 
We shall explain it in detail. 

For simplicity we only treat the local factor for an odd prime $p$. Let $\Sym^n(\Z_p)^*$ be the set of half integral symmetric matrices of size $n$. For $N \in \Sym^n(\Z_p)^*$ with $\det N \ne 0$, we denote the Siegel series of $N$ at $p$ by $S_{n,p}(N,s)$, here $s$ is a complex variable.
Let $N_0 = \diag(\alpha_1 p^{u_1}, \ldots, \alpha_{n-1} p^{u_{n-1}})$ with $\alpha_i \in \Z_p^\times$ and $0 \le u_1 \le \cdots \le u_{n-1} \in \Z$. For $\alpha_n \in \Z_p^\times$ and  $u_n \ge u_{n-1}$ we set
$N = N_0 \perp \langle \alpha_n  p^{u_n} \rangle$ and $N' = N_0 \perp \langle \alpha_{n}  p^{u_{n}+2} \rangle$. Then Katsurada found the inductive relation (\cite[Theorem 2.6]{Kat})
\begin{equation} \label{inductive_relation_Katsurada_1}
S_{n,p}(N',s) -p^{n+1-2s}S_{n,p}(N,s) = (1-p^{-s})(1+p^{1-s}) S_{n-1,p}(N_0,s-1). 
\end{equation}
The other pillar, the functional equation of Siegel series, is written as follows. Let $F_{n,p}(N,s)$ be the \emph{polynomial part} of Siegel series, that is removed suitable Euler factors of Riemann zeta functions or  a Dirichlet $L$-function from $S_{n,p}(N,s)$ (cf.\ \cite[Proposition 3.6]{Shi2}). Then Katsurada proved the functional equation of the form  (\cite[Theorem 3.2]{Kat})
\[ F_{n,p}(N,n+1-s) = R(N,s) p^{(s-\frac{n+1}{2})d_p(N)} F_{n,p}(N,s), \]
here $d_p(N) = \ord_p(\det N)$ and the remaining term $R(N,s)$  satisfies $R(N,s) = R(N',s)$.
Thus change the variable $s \mapsto n+1-s$ in (\ref{inductive_relation_Katsurada_1}) and apply functional equations, one has the relation of the form
\begin{equation} \label{inductive_relation_Katsurada_2}
S_{n,p}(N',s)_p-S_{n,p}(N,s) = H_n(N,s) S_{n-1,p}(N_0,s).
\end{equation}
Combining (\ref{inductive_relation_Katsurada_1}) and (\ref{inductive_relation_Katsurada_2}), Katsurada gave the recursion formula of Siegel series  of the form
\[ F_{n,p}(N,s) = C(N,s)F_{n-1,p}(N,s-1) + D(N,s)F_{n-1,p}(N,s) \]
(\cite[Theorem 4.1]{Kat}) and got an explicit formula of $F_{n,p}(N,s)$ .

\bigskip

Our interest is the Fourier coefficients of Siegel Eisenstein series with level and characters. The Euler $p$-factor for a prime $p$, that divides the level, is called the \emph{ramified Siegel series} and its behavior is quite different from the ordinary Siegel series. 
We consider the case that the level is  an odd prime $p$ and the character is the trivial character $\chi_0$ of modulo $p$ or the quadratic character $\chi_p$ of modulo $p$. Then the space of Siegel Eisenstein series is $(n+1)$-dimensional for degree $n$ case. We define $w_\nu \in \Sp(n,\Z)$ ($0 \le \nu \le n$) as in (\ref{def_w_nu}), that corresponds to each $0$-dimensional cusp. For $\psi = \chi_0$ or $\chi_p$, the Siegel Eisenstein series $E^n_{k,\psi}(w_\nu; Z,s)$ of weight $k$, character $\psi$ attached to $w_\nu$ is defined as in (\ref{def_Eisenstein_w_nu}).  The ramified Siegel series of $E^n_{k,\psi}(w_\nu;Z,s)$ at $N \in \Sym^n(\Z_p)^*$ is denoted by $S^{w_\nu}_n(\psi,N,s)$. Since $\{ E^n_{k,\psi}(w_\nu;Z,s) \}_{0 \le \nu \le n}$ form a basis of the space of Siegel Eisenstein series, our aim is to give a formula of $S^{w_\nu}_n(\psi,N,s)$ for all $\nu$.

After several concrete calculations for small degree case (cf.\ \cite{Mizno}, \cite{Ta}, \cite{Di}, \cite{Gu2}, \cite{Gu3}), the author (\cite{Gu1}) and Watanabe (\cite{Wat1}) independently have found formulas of ramified Siegel series  for general degree case, both results are based on the paper \cite{SH} by Sato and Hironaka. Here we call these results  Sato-Hironaka-type formula.
Note that the author computed ramified Siegel series only attached to cusp $w_0$, but Watanabe treated all the cusps in more general situation. Unfortunately however, these results are too complicated to handle directory. 
The reason is as follows.  Siegel series are defined  by an integral over $\Sym^n(\Q_p)$, the space of symmetric matrices over $\Q_p$. In \cite{SH} Sato and Hironaka computed the integral by decomposing $\Sym^n(\Q_p)$ to the orbits for the action of $\G \subset \GL_n(\Z_p)$, that is the Iwahori subgroup corresponding to the standard Borel subgroup. However the number of the orbits is so large. In order to compute $S^{w_\nu}_n(\psi,N,s)$, first we choose $\sigma$ in the symmetric group of degree $n$ such that $\sigma^2 = 1$, next we have to consider all the $\sigma$-stable partitions $\{1,2, \ldots, n \} = \bigcup_{i=0}^r I_i$ such that $\sum_{i=t}^n \sharp I_i = \nu$ for some $t$ ($0 \le t \le r+1$). The number of such partitions grows quite rapidly as $n$ increases. 
Therefore we need to find more useful formula, in order to compute the ramified Siegel series for higher degree case.

\bigskip

Recently there were two developments for our problem. One corresponds to the functional equations; the author found the functional equations of Siegel Eisenstein series of level $p$ in a simple form (\cite{Gu4}). For that we consider the $U(p)$-operator, that is a kind of Hecke operator, acting on the space of Siegel Eisenstein series. The eigenvalues are given by $p^{l(\nu,s)}$ ($0 \le \nu \le n$) with $l(\nu,s) = n(k-s)/2+s\nu -\nu(\nu+1)/2$. 
For each $\nu$, we denote the Siegel Eisenstein series, which is a $U(p)$-eigenfunction with eigenvalue $p^{l(s,\nu)}$ and suitably normalized, by $E^{(\nu)}_{k,\psi}(Z,s)$  (see (\ref{U(p)-eigen_Eisenstein})). Then there exists a functional equation between $E^{(\nu)}_{k,\psi}(Z,s)$ and $E^{(n-\nu)}_{k,\psi}(Z,n+1-s)$.
Hence if we write the ramified Siegel series of $E^{(\nu)}_{k,\psi}(Z,s)$  by $S^{(\nu)}_n(\psi,N,s)$ and call it \emph{$U(p)$-characteristic ramified Siegel series}, then we can write the functional equations between $S^{(\nu)}_n(\psi,N,s)$ and $S^{(n-\nu)}_n(\psi,N,n+1-s)$ explicitly. (Theorem \ref{thm_FE_ramified_Siegel_series}).

The other development, corresponding to inductive relations, is mainly due to Watanabe. We note that the inductive relation of ramified Siegel series for trivial character is easily deduced from the Katsurada's inductive relation of the ordinary Siegel series.
For the quadratic character case, Watanabe proved in \cite{Wat2} the inductive relation of ramified Siegel series similar to (\ref{inductive_relation_Katsurada_2});  he found
\begin{equation} \label{intro_inductive_formula_ramified}
S^{w_\nu}_n(\chi_p,N',s) - S^{w_\nu}_n(\chi_p,N,s) = H_n(\chi_p,N,s) S_{n-1}^{w_\nu}(\chi_p,N_0,s), 
\end{equation}
with a certain function $H_n(\chi_p, N,s)$, that is independent of $\nu$. We note that for the proof of the above formula, Watanabe used his Sato-Hironaka-type formula of the ramified Siegel series given in \cite{Wat1}.
By using (\ref{intro_inductive_formula_ramified}), Watanabe found the recursion formula of a certain ramified Siegel series, that is quite similar to Katsurada's recursion formula. However since $H_n(\chi_p,N,s)$ is independent of $\nu$, we see that (\ref{intro_inductive_formula_ramified}) also works for each $U(p)$-characteristic ramified Siegel series (see (\ref{inductive_relation_eigen-ramified})), thus together with functional equations we can get the recursion formula of $S^{(\nu)}_n(\chi_p,N,s)$ for all $\nu$.

\bigskip

Now we shall introduce the contents of this paper and state our main theorem. In section \ref{Sec_Siegel_Eisenstein} we study the transformation matrix between $\{ S^{w_\nu}_n(\chi_p,N,s) \}$ and $\{S^{(\nu)}_n(\chi_p,N,s) \}$ explicitly. We  write
\[  S^{w_\nu}_n(\psi,N,s) = \sum_{r = \nu}^n c_{\psi}(s)_{\nu r} \, S^{(r)}_{n}(\psi,N,s) \]
then if $r \equiv \nu \bmod 2$ then
\begin{align*} c_{\chi_0}(s)_{\nu r} & =  \frac{\displaystyle p^{\frac{(r-\nu)(r-\nu+2)}{4} -(r-\nu)s}(1-p^{\nu-s}) \prod_{k=\nu+1}^{r}(p^k-1)}{\displaystyle (1-p^{r-s}) \prod_{k=1}^{(r-\nu)/2} (1-p^{\nu+r-1+2k-2s})(p^{2k}-1)}, \\
c_{\chi_p}(s)_{\nu r} & = \frac{\displaystyle \chi_p(-1)^{\frac{r-\nu}{2}} p^{\frac{r-\nu}{2}(\frac{r-\nu}{2}-2s)} \prod_{k=\nu+1}^{r} (p^k-1)}{\displaystyle \prod_{k=1}^{(\nu-r)/2} \bigl(1-p^{\nu+r-1+2k-2s} \bigr)(p^{2k}-1)},
\end{align*}
if $r \not\equiv \nu \bmod 2$ then
\[c_{\chi_0}(s)_{\nu r}  = \frac{\displaystyle p^{\frac{(r-\nu)^2-1}{4}-(r-\nu)s} \prod_{k=r+1}^{\nu}(p^k-1)}{ \displaystyle (1-p^{r-s}) \prod_{k=1}^{(r-\nu-1)/2} (1-p^{\nu+r+2k-2s})(p^{2k}-1)} \]
and $c_{\chi_p}(s)_{\nu r}  = 0$ (Proposition \ref{prop_explicit_c_psi_ij}). In section \ref{Sec_functional_equation} we prove the functional equations of ramified Siegel series, using the functional equations of Siegel Eisenstein series. As a corollary we also find the explicit form of $S^{(0)}_n(\psi,N,s)$. 
In section \ref{Sec_recursion}, we show inductive relations using the result of Watanabe, consequently we can get the recursion formula of ramified Siegel series.

In order to state the main theorem we prepare some notations (see section \ref{Sec_functional_equation} for detail).
For $N \in \Sym^n(\Z_p)^*$ with $\det N \ne 0$, we put $D_N = (-4)^{[n/2]} \det N$ and $d_p(N) = \ord_p D_N$. In the case that $n$ is even, let $\chi_N$ be the quadratic or trivial character corresponding to the quadratic extension $\Q(\sqrt{D_N})/\Q$ and  $\chi_N^*$ the primitive character arising from $\chi_N \chi_p$. We set $a_N \in \{0, 1 \}$ so that $a_N \equiv d_p(N) \bmod 2$. On the other hand in the case that $n$ is odd, we write 
\[ \zeta_p(N) = h_p(N) \, \bigl( \det N, (-1)^{\frac{n-1}{2}} \det N \bigr)_p, \]
where $h_p(N)$ is the Hasse invariant of $N$ and $( \cdot \, , \, \cdot)_p$ is the Hilbert symbol. We define $D'_N$ by $D_N = p^{d_p(N)} D'_N$ and put $\eta_p(N) = \chi_p(D_N') \zeta_p(N)$. 
Now let $N = \diag(\alpha_1 p^{u_1}, \ldots, \alpha_{n} p^{u_{n}})$ with $\alpha_i \in \Z_p^\times$ and $0 \le u_1 \le \cdots \le u_n$. We set $N_0 = \diag(\alpha_1 p^{u_1}, \ldots, \alpha_{n-1} p^{u_{n-1}})$. 
We can calculate the function $H_n(\psi,N,s)$, that appears in our inductive relations of the type (\ref{inductive_relation_Katsurada_2}) or (\ref{intro_inductive_formula_ramified}) explicitly by using the formula of $S^{(0)}_n(\psi,N,s)$ proved in section \ref{Sec_functional_equation}.  For the trivial character case, if $n$ is even then 
\[ H_n(N,s)  = \zeta_p(N_0) p^{\frac{1}{2} d_p(N_0) + (\frac{n+1}{2}-s)(u_n - a_N + 2)} \frac{(1-p^{n-2s})(1-\chi_N(p) p^{s-\frac{n}{2} -1})}{1-\chi_N(p)p^{\frac{n}{2}-s}},\]
and if $n$ is odd then
\[ H_n(N,s)  = -\zeta_p(N) p^{\frac{1}{2} d_p(N) + (\frac{n}{2}-s)(u_n - a_{N_0} + 2)} \frac{(1-p^{n+1-2s})(1-\chi_{N_0}(p) p^{s-\frac{n-1}{2}})}{1-\chi_{N_0}(p)p^{\frac{n+1}{2}-s}}
\]
For the quadratic character case, if $n$ is even then 
\[ H_n(\chi_p, N,s)  =  \ve_p \eta_p(N_0) p^{\frac{1}{2} d_p(N_0) + (\frac{n+1}{2}-s)(u_n + a_N+1)} \dfrac{ (1-p^{n-2s})(1-\chi_N^*(p)p^{s-\frac{n}{2}-1})}{1-\chi_N^*(p)p^{\frac{n}{2}-s}},
\]
and if $n$ is odd then	
\[ H_n(\chi_p,N,s)  = -\ve_p \eta_p(N) p^{\frac{1}{2}d_p(N) + (\frac{n}{2}-s)(u_n + a_{N_0}+1 )} \dfrac{ (1-p^{n+1-2s})(1-\chi_{N_0}^*(p)p^{s-\frac{n-1}{2}})}{1-\chi_{N_0}^*(p)p^{\frac{n+1}{2}-s}}.
\]
We understand that $H_n(\psi,N,s) = H_n(N,s)$ for $\psi = \chi_0$ case. Under the above notation we have the main theorem.

\begin{theorem}[{Theorem \ref{thm_Main_theorem_1}, \ref{thm_Main_theorem_2}}] We have the following recursion formula of the $U(p)$-characteristic ramified Siegel series.
\[ S_n^{(\nu)}(\psi,N,s) = \frac{1-p^{\nu+1-2s}}{1-p^{n+1-2s}} S_{n-1}^{(\nu-1)}(\psi,N_0,s-1) - \frac{H_n(\psi,N,s)}{1-p^{n+1-2s}} S_{n-1}^{(\nu)}(\psi,N_0,s) . \]
\end{theorem}

Finally in section \ref{Sec_example} we calculate examples for degree 2 and 3 case, that coincide with the previous results \cite{Gu2} and \cite{Gu3}.

\begin{acknowledgment} The author would thank to Masahiro Watanabe in Kyoto University for many useful discussions and comments.
\end{acknowledgment}

\begin{notation}
Throughout this paper $p$ stands for a fixed odd prime. An arbitrary prime number is denoted by $q$. We put $\chi_0$ the trivial Dirichlet character modulo $p$, and $\chi_p = (\frac{\cdot}{p})$  the quadratic Dirichlet character modulo $p$. 
For a commutative ring $A$ of characteristic zero, we put $\Sym^n(A)$ the set of symmetric matrices of size $n$ whose coefficients belong to $A$, and $\Sym^n(A)^*$ the set of half integral symmetric matrices, whose diagonal entries are contained in $A$ and off-diagonal entries are contained in $\frac{1}{2}A$.
We denote the zero-matrix of size $n$ by $0_n$ and  the identity matrix of size $n$ by $\bfI_n$.
Let $\Sp(n,\Z)$ be the symplectic modular group of rank $n$, matrix size $2n$ and $P_n(\Z)$ be its subgroup consist of the elements whose lower-left $(n \times n)$ blocks are zero matrices. We put
\[ \G_0^n(p) = \left\{ \begin{pmatrix} A & B \\ C & D \end{pmatrix} \in \Sp(n,\Z) \biggm|  C \equiv 0 \bmod p \right\}. \]
For $R \in M_n(\Q)$ we put $\epi(R) = \exp (2 \pi \sqrt{-1} \Tr(R))$. Similarly for $R \in M_n(\Q_q)$ we also write $\epi(R) = \exp (2 \pi \sqrt{-1} \varphi(\Tr(R)))$, here $\varphi$ is the  natural identity map $\varphi \colon \Q_q/\Z_q \simeq \bigcup_{n \ge 0} \frac{1}{q^n} \Z/\Z$. For a non-negative integer $r$ we set $\Gamma_r(s) = \pi^{r(r-1)/4} \prod_{j=0}^{r-1} \Gamma(s-j/2)$. We understand $\Gamma_0(s) = 1$.
\end{notation}

\section{Siegel Eisenstein series of level $p$} \label{Sec_Siegel_Eisenstein}

\subsection{Ramified Siegel series}

For $0 \le \nu \le n$, we put 
\begin{equation} \label{def_w_nu}
w_\nu = \left( \begin{array}{cc|cc} 0_{\nu} & & -\bfI_\nu & \\ & \bfI_{n-\nu} & & 0_{n-\nu} \\ \hline \bfI_\nu & & 0_\nu& \\  & 0_{n-\nu} & & \bfI_{n-\nu} \end{array} \right) \in \Sp(n,\Z). 
\end{equation}
Then $\{ w_\nu \}_{0 \le \nu \le n}$ form a representative set of $\G_0^n(p) \backslash \Sp(n,\Z)/ P_n(\Z)$. 
We regard a Dirichlet character $\psi$ modulo $p$ as a character of  $\G^n_0(p)$  by $\psi(\g) = \psi(\det D)$ for $\g = \begin{pmatrix} A & B \\ C& D \end{pmatrix} \in \G^n_0(p)$,  and define the function $\psi_{w_\nu}$ on  $P_{n}(\Z) w_\nu \G^n_0(p)$ by 
\[  \psi_{w_\nu}(\g) =\psi(\eta) \psi(\kappa), \quad  \g = \eta w_\nu \kappa, \ \eta \in P_{n}(\Z), \ \kappa \in \G_0^n(p). 
\]
Let $\ha_n$ be the Siegel upper half space. For $\g = \begin{pmatrix} A& B \\ C & D \end{pmatrix} \in \Sp(n,\Z)$ and $Z \in \ha_n$, we put $j(\g, Z) = \det (CZ+D)$. 

From now on we write $\psi = \chi_0$ or $\chi_p$. For an integer $k$ such that $\psi(-1) = (-1)^k$ and $s \in \C$ we define the Siegel Eisenstein series of level $p$, weight $k$ with character $\psi$ attached to cusp $w_\nu$ as follows. For $Z = X + \sqrt{-1}Y \in \ha_n$,
\begin{equation} \label{def_Eisenstein_w_nu}
 E^n_{k,\psi}(w_\nu;Z,s) = \det(Y)^{(s-k)/2}  \sum_{\g} \psi_{w_\nu}(\g) \,  j(\g, Z)^{-k} \, |j(\g,Z)|^{-s+k},
\end{equation}
here $\g$ runs through the set $P_n(\Z) \backslash P_n(\Z) w_\nu \G^n_0(p)$. The right hand side converges when $\re(s) > n+1$. 
For $k > n+1$, we simply write $E^n_{k,\psi}(w_\nu;Z) = E^n_{k,\psi}(w_\nu,Z,k)$, that is contained in the space of holomorphic Siegel modular forms.  In particular
\[ E^n_{k,\psi}(w_0;Z) = \sum_{ \begin{smatrix} A & B \\ C & D \end{smatrix} \in P_n(\Z) \backslash \G^n_0(p)} \psi(\det D) \det(CZ + D)^{-k} \]
is the usual Siegel Eisenstein series of degree $n$, level $p$ with character $\psi$.

Let 
\[ E^n_{k,\psi}(w_\nu;Z,s) = \sum_{N \in \Sym^n(\Z)^*} A_{\psi}^{w_\nu}(N;Y,s) \epi(NX) \]
be the Fourier expansion of $E^n_{k,\psi}(w_\nu;Z,s)$. For holomorphic Eisenstein series, we can write
\[ E^n_{k,\psi}(w_\nu;Z) = \sum_{\substack{N \in \Sym^n(\Z)^* \\ N \ge 0} } A_\psi^{w_\nu}(N,k) \epi(NZ). \]

Our aim is to write down the Fourier coefficient $A_\psi^{w_\nu}(N,k)$ explicitly for each $\nu$.  It is known that $A_{\psi}^{w_\nu}(N,k)$ has an Euler product expression. In order to explain it, we define the Siegel series as follows. For a prime number $q$, let
\[ \mathcal{M}_n(q) = \left\{ (C,D) \in M_{n,2n}(\Z_q) \biggm| \begin{array}{l} \text{$(C,D)$ is symmetric and co-prime,} \\ \text{$\det C = q^l$ for some $l \ge 0$} \end{array}\right\}. \]
Any $T \in \Sym^n(\Q_q)$ is written by $T = C^{-1} D$ with $(C,D) \in \mathcal{M}_n(q)$, unique up to the $\SL_n(\Z_q)$ action from the left. We put $\delta_q(T) = \det C$. Then the (ordinary) Siegel series at $q$ is defined by
\[ S_{n,q}(N,s)  = \sum_{R \in \Sym^n(\Q_q/\Z_q)} \delta_q(R)^{-s} \epi(RN). \]
Here we shortly denote $\Sym^n(\Q_q) \bmod \Sym^n(\Z_q)$ by $\Sym^n(\Q_q/\Z_q)$. Since this series is a rational function in $q^{-s}$, by abuse of language we also write it $S_{n,q}(N,q^{-s})$ or simply $S_n(N,q^{-s})$.

For the $p$-local factor, we need more notations. For $0 \le \nu \le n$ we put
\begin{align*}
\mathcal{M}_n^\nu(p) & = \{(C,D) \in \mathcal{M}_n(p) \mid \rank(C \bmod p) = \nu \}, \\
\Sym^n(\Q_p)^{(\nu)} & = \{R = C^{-1} D \in \Sym^n(\Q_p) \mid (C,D) \in \mathcal{M}_n^\nu(p) \}.
\end{align*}
Let $R = C^{-1} D \in \Sym^n(\Q_p)^{(\nu)}$. Then there exists $\g = \begin{smatrix} * & * \\ C & D \end{smatrix} \in P_n(\Z_p) w_\nu K_p$, here
\[ K_p = \Bigl\{ \begin{pmatrix} A& B \\ C & D \end{pmatrix} \in \Sp(n,\Z_p) \biggm| C \equiv 0 \bmod p \Bigr\}. \]
We set $\widetilde{\psi}_{w_\nu}(R) = \psi_{w_\nu}(\g)$.

\begin{definition}
For $N \in \Sym^n(\Q_p)^*$, we define the \emph{ramified Siegel series at $w_\nu$} by
\begin{align*}
S_n^{w_\nu}(\psi,N,s) & = \sum_{R \in \Sym^n(\Q_p/\Z_p)^{(\nu)}} \widetilde{\psi}_{w_\nu}(R) \delta_p(R)^{-s} \epi(NR).
\end{align*}
It is also a rational function in $p^{-s}$, we may write it as $S^{w_\nu}_n(\psi, N, p^{-s})$.
\end{definition}

The Fourier coefficients of Siegel Eisenstein series have Euler product expressions as follows.

\begin{proposition}[{cf.\ \cite[Proposition 3.1, Lemma 3.3]{Gu4}}] \label{prop_Fourier_exp_Eisenstein_level_p_1}
Let $N \in \Sym^n(\Z)^*$ and $\det N \ne 0$. Then the Fourier coefficient  of $E^n_{\psi,k}(w_\nu;Z,s)$ at $N$ is given by
\begin{align*} A_\psi^{w_\nu}(N;Y,s) & =  \det(Y)^{(s-k)/2} \Xi_n \Bigl(Y,N, \frac{s+k}{2},\frac{s-k}{2} \Bigr) \\
& \qquad \times  S_n^{w_\nu}(\psi, N,p^{-s}) \prod_{q \ne p} S_{n,q}(N, \psi(q)q^{-s}). 
\end{align*}
Here $\Xi_n(Y,N, \alpha, \beta)$ is the confluent hypergeometric function, investigated by Shimura \emph{\cite{Shi1}}.
\end{proposition}

\begin{remark} In \cite{Shi1} the confluent hypergeometric function above is denoted by $\xi_n(Y,N, \alpha, \beta)$. We use the letter $\Xi$ instead of $\xi$ in order to avoid the confusion to the completed Riemann zeta function or the Dirichlet $L$-functions.
\end{remark}

In particular, by \cite[(4,34K), (4,35K)]{Shi1}, we have the following result.
\begin{corollary} Let $N \in \Sym^n(\Z)^*$ and $N> 0$. Then the Fourier coefficient $A^{w_\nu}_{\psi}(N,k)$ of $E^n_{k,\psi}(w_\nu;Z)$ is given by
\[ \frac{2^{\frac{-n(n-1)}{2}} (-2 \pi \sqrt{-1})^{nk} (\det N)^{k-\frac{n+1}{2}}}{\Gamma_n(k)}  S_n^{w_\nu}(\psi, N,p^{-k}) \prod_{q \ne p} S_{n,q}(N, \psi(q)q^{-k}). \]

\end{corollary}

An explicit formula of $S_{n,q}(N,q^{-s})$ is given by Katsurada (\cite{Kat}). Thus the remaining problem is to give an explicit formula of $S_n^{w_\nu}(\psi,N,s)$.

\subsection{$U(p)$-eigenfunction in the space of Siegel Eisenstein series}

Let 
\[ \mathcal{E}_{k,s}(\G^n_0(p),\psi)= \langle  E^{n}_{k,\psi}(w_\nu;Z,s) \mid 0 \le \nu \le n \rangle_\C \]
be the space of Siegel Eisenstein series. The Hecke operator $U(p)$ acts on this space as follows. For
\[ E(Z,s) = \sum_{N \in \Sym^n(\Z)^*} A(N;Y,s) \epi(NX) \in \mathcal{E}_{k,s}(\G^n_0(p),\psi), \]
we have
\[ U(p) E(Z,s) = \sum_{N \in \Sym^n(\Z)^*} A(pN;p^{-1}Y,s) \epi(NX). \]
For more precise explanation of $U(p)$, please refer \cite[section 2]{Gu4}. 

Eigenfunctions of the $U(p)$-operator play an important role in our argument. We quote necessarily properties from \cite{Gu4}.

We put
\begin{equation} l_n(s,\nu) = l(s,\nu) = \frac{n(k-s)}{2} + s\nu - \frac{\nu(\nu+1)}{2}. \label{eq_l(s,nu)}
\end{equation}
Then the eigenvalues of $U(p)$ on $\mathcal{E}_{k,s}(\G^n_0(p),\psi)$ are given by  $p^{l(s,\nu)}$ with $0 \le \nu \le n$ (\cite[Corollary 2.5]{Gu4}).

For $0 \le i,j \le n$ we define
\[ m_{\chi_0}(s)_{ij}  = 
 \begin{cases} \dfrac{\displaystyle p^{si -j(j+1)/2 + (j-i)(j-i+2)/4}  \prod_{r=i+1}^j(p^r-1)}{\displaystyle \prod_{r=1}^{(j-i)/2} (p^{2r}-1)}  &  \begin{array}{l} \text{if $i \le j$}, \\ \text{$j -i$ is even,} \end{array} \\[20pt]
\dfrac{ \displaystyle p^{si -j(j+1)/2 + ((j-i)^2-1)/4}\prod_{r=i+1}^j(p^r-1)}{\displaystyle \prod_{r=1}^{(j-i-1)/2} (p^{2r}-1)} &
\begin{array}{l} \text{if $i \le j$}, \\ \text{$j-i$ is odd}, \end{array} \\[20pt]
0 & \text{if $i>j$.} \end{cases} 
\]
\[ m_{\chi_p}(s)_{ij}  = \begin{cases} \dfrac{\displaystyle \chi_p(-1)^{(j-i)/2}p^{si-j(j+1)/2 + (j-i)^2/4} \prod_{k=i+1}^j(p^k-1)}{\displaystyle \prod_{k=1}^{(j-i)/2} (p^{2k}-1)} & \begin{array}{l} \text{if $i \le j$}, \\ \text{$j-i$ is even,} \end{array} \\[20pt]
 0 & \text{otherwise}. \end{cases} \]
Then for $\psi = \chi_0$ or $\chi_p$,
\[ U(p)E^n_{k,\psi}(w_i;Z,s) = p^{n(k-s)/2} \sum_{j=i}^n m_{\psi}(s)_{ij} E^n_{k,\psi}(w_j;Z,s). \]

In particular the representative matrix of $U(p)$ is given by a triangular matrix. In order to find eigenfunctions,
we define $b_{\chi_p}(s)_{ij}$ for $0 \le i,j \le n$  inductively as follows.
\begin{align}
&\text{if $i>j$} &   b_{\psi}(s)_{ij} &= 0 \notag \\ 
&\text{if $i=j$} &  b_{\psi}(s)_{ii} &= 1 \label{inductive_b_psi}    \\
&\text{if $i<j$} &  b_{\psi}(s)_{ij} & = \frac{p^{-js + j(j+1)/2}}{p^{(j-i)((j+i+1)/2-s)}-1} \sum_{k=i}^{j-1} m_{\psi}(s)_{kj} b_{\psi}(s)_{ik}. \notag
\end{align}
Note that $b_{\chi_p}(s)_{ij} = 0$ if $i \not\equiv j \bmod 2$.
Then
\begin{equation} \label{U(p)-eigen_Eisenstein}
E^{n,(\nu)}_{k,\psi}(Z,s) = E^{n}_{k,\psi}(w_\nu;Z,s) + \sum_{r= \nu +1}^n b_{\psi}(s)_{\nu r} \, E^n_{k,\psi}(w_r;Z,s) 
\end{equation}
is the eigenfunction of the $U(p)$-operator with eigenvalue $p^{l(s,\nu)}$ (\cite[Proposition 2.10]{Gu4}).

Let $B_{\psi}(s) = \bigl( b_{\psi}(s)_{ij} \bigr)_{0 \le i,j \le n}$ be the $((n+1) \times (n+1))$-matrix and  we denote the $(i,j)$-component of its inverse matrix $B_{\psi}(s)^{-1}$ by $c_{\psi}(s)_{ij}$.  Thus we have
\[ E^n_{k,\psi}(w_\nu;Z,s) = E^{n,(\nu)}_{k,\psi}(Z,s) + \sum_{ r=\nu + 1 }^n c_{\psi}(s)_{\nu r} \, E^{n,(r)}_{k,\psi}(Z,s). \]
Note that for the quadratic character case, $c_{\chi_p}(s)_{\nu r} = 0$ unless $\nu \equiv r \bmod 2$.

\begin{remark}
By \cite[Remark after Lemma 3.5]{Gu4} we have
\[ c_{\psi}(s)_{ij} = S^{w_i}_j(\psi,0_j,s), \]
that is, $c_{\psi}(s)_{ij}$ is the ramified Siegel series of degree $j$ with respect to $w_i$ for the zero-matrix. 
\end{remark}

First we shall show the explicit formulas of $b_{\psi}(s)_{ij}$ and $c_\psi(s)_{ij}$. For that we need the following lemma.

\begin{lemma} \label{lem_product_sum_2} Let $x$ and $y$ are indeterminates. Then we have
\[ \sum_{t=1}^r \prod_{k=1}^t \frac{(x^{k-1}y-1)(x^{r+1-k}-1)}{(x^{r-k}-y)(x^k-1)} = -1. \]
\end{lemma}

\begin{proof}
For $1 \le k \le r$, let
\[ G_k = \dfrac{(x^{k-1}y-1)(x^{r+1-k}-1)}{(x^{r-k}-y)(x^k-1)} \quad \text{and} \quad  H_k = \dfrac{(x^{k-1}y-1)(x^{r+1-k}-1)}{(1-y)(x^r-1)}. \]
Then $G_r = H_r$ and
\[ 1 + H_k = \frac{(x^{r+1-k}-y)(x^{k-1}-1)}{(1-y)(x^r-1)}, \]
thus $G_{k-1}(1+H_k) = H_{k-1}$. It shows
\[ \prod_{k=1}^{r-1} G_k + \prod_{k=1}^r G_k = \left( \prod_{k=1}^{r-2} G_k \right) G_{r-1} ( 1 + H_r) = \left( \prod_{k=1}^{r-2} G_k \right) H_{r-1}. \]
Repeating the same argument we conclude
$\sum_{t=1}^r \prod_{k=1}^t G_r = H_1 = -1$,	
that proves our assertion.
\end{proof}

\begin{lemma} \label{lem_product_sum_1}
For indeterminates $x$ and $y$ we have
\[ \sum_{t=1}^r \prod_{k=1}^t \frac{(y-x^{k-1})(x^{r+1-k}-1)}{x^k-1} = y^r-1. \]

\end{lemma}

\begin{proof} We write $\displaystyle F_t(y,r) = \prod_{k=1}^t \dfrac{(y-x^{k-1})(x^{r+1-k}-1)}{x^k-1}$. Since $\sum_{t=1}^r F_t(y,r)$ is a polynomial in $y$ of degree $r$, it suffices to show that 
\[ \sum_{t=1}^r F_t(x^n,r) = x^{nr}-1, \quad  \text{for all $n \in \Z_{\ge 0}$}. \]
Note that $F_t(x^n,r) = 0$ for $t>n$. We prove the above assertion by induction on $n$. The case of $n=0$; $\sum_{t=1}^r F_t(1,r) = 0$ is trivial. For $n \ge 1$, we rewrite the numerator of $F_t(x^n,r)$ as 
\begin{align*}  & (x^r-1)(x^n-1)  x^{t-1} \prod_{k=2}^t (x^{n-1}-x^{k-2})(x^{r+1-k}-1) \\
& = (x^r-1)(x^n-1) x^{t-1} \prod_{k=1}^{t-1} (x^{n-1}-x^{k-1})(x^{r-k}-1).
\end{align*}
Decompose $x^n-1 = x^{n-t}(x^t-1) + (x^{n-t}-1)$, we have
\begin{align*}
F_t(x^n,r) & = (x^r-1)x^{n-1} \prod_{k=1}^{t-1} \frac{(x^{n-1}-x^{k-1})(x^{r-k}-1)}{x^k-1} \\
& + (x^r-1)(x^{n-1}-x^{t-1}) \frac{\displaystyle  \prod_{k=1}^{t-1} (x^{n-1}-x^{k-1})(x^{r-k}-1)}{\displaystyle \prod_{k=1}^t (x^k-1)} \\
& = (x^r-1)x^{n-1} F_{t-1}(x^{n-1}, r-1) + F_t(x^{n-1}, r).
\end{align*} 
Here we understand $F_0(y,r-1) = 1$. Then by the induction hypothesis we have
\begin{align*}
\sum_{t=1}^r F_t(x^n,r) & = (x^r-1)x^{n-1} \Bigl(1+ \sum_{t=2}^r F_{t-1}(x^{n-1},r-1) \Bigr) + \sum_{t=1}^r F_t(x^{n-1},r) \\
& = x^{nr}-1,
\end{align*}
thus the lemma is proved.
\end{proof}

\begin{remark} This lemma is an easy consequence of $q$-analogue of Chu-Vandermonde theorem, that is pointed out by Watanabe. We define the $q$-Pochhammer symbol $(a;q)_n$ by $(a;q)_n = \prod_{i=1}^n(1-aq^{i-1})$. The $q$-hypergeometric series is defined as
\[ \leftidx_{2}{\phi}_1 \left( \! \! \begin{array}{c} a, b \\ c \end{array}; q, z \right) = \sum_{k=0}^\infty \frac{(a;q)_k \, (b;q)_k}{(c;q)_k \,(q;q)_k} z^k. \]
Then the theorem of $q$-Chu-Vandermonde says that
\begin{equation} \label{q-Chu-Vandermonde}
 \leftidx_{2}{\phi}_1 \left( \! \! \begin{array}{l} a, q^{-n} \\ \ \, c \end{array} \! \! ; q,q \right) = \frac{(c/a;q)_n}{(c;q)_n} a^n 
\end{equation}
holds (\cite[(1.5.3)]{GR}). We can rewrite the formula of Lemma \ref{lem_product_sum_1} as
\[ \sum_{t=0}^r \frac{(y;x^{-1})_t \, (x^r;x^{-1})_t}{(x^{-1}, x^{-1})_t} (x^{-1})^t = y^r. \]
Substituting $a=y$, $q = x^{-1}$ and $c=0$ in (\ref{q-Chu-Vandermonde}), we get our lemma. \qed
\end{remark}

\begin{proposition} \label{prop_explicit_b_psi} 

\begin{enumerate}[$(1)$]
\item Let $\psi = \chi_0$. For $i \le j$ we have
\[ b_{\chi_0}(s)_{ij}  = \begin{cases} \frac{\displaystyle (-1)^{(j-i)/2} p^{-(j-i)s} (1-p^{j+1-s}) \prod_{k=i+1}^{j} (p^k-1)}{\displaystyle (1-p^{i+1-s})  \prod_{k=1}^{(j-i)/2} (1-p^{2i+1+2k-2s})(p^{2k}-1)} & \text{if $i \equiv j \bmod 2$,} \\
\frac{\displaystyle (-1)^{(j-i+1)/2} p^{-(j-i)s} \prod_{k=i+1}^{j} (p^k-1)}{\displaystyle (1-p^{i+1-s})\prod_{k=1}^{(j-i-1)/2}(1-p^{2i+1+2k-2s})(p^{2k}-1)} & \text{if $i \not\equiv j \bmod 2$.} \end{cases}
 \]

\item Let $\psi = \chi_p$. For $i \le j$ and $i \equiv j \bmod 2$ we have
\begin{equation*} 
b_{\chi_p}(s)_{ij} = \frac{\displaystyle \bigl( -\chi_p(-1) \bigr)^{\frac{j-i}{2}} p^{\frac{j-i}{2}(1-2s)} \prod_{k=i+1}^{j}(p^k-1)}{ \displaystyle \prod_{k=1}^{(j-i)/2} (1-p^{2i+1+2k-2s})(p^{2k}-1)}.
\end{equation*}

\end{enumerate}
\end{proposition}

\begin{proof}
We only explain the assertion (1), since (2) can be proved similarly. We rewrite our assertion as
\begin{align} 
b_{\chi_0}(s)_{i,i+2r} & = \frac{\displaystyle (-1)^r p^{-2rs} (1-p^{i+2r+1-s}) \prod_{k=i+1}^{i+2r} (p^k-1)}{\displaystyle (1-p^{i+1-s}) \prod_{k=1}^r (1-p^{2i+1+2k-2s})(p^{2k}-1)},  \label{rewrite_b_chi_0_even}\\
b_{\chi_0}(s)_{i,i+2r+1} & = \frac{\displaystyle (-1)^{r+1} p^{-(2r+1)s} \prod_{k=i+1}^{i+2r+1} (p^k-1)}{\displaystyle (1-p^{i+1-s})\prod_{k=1}^r(1-p^{2i+1+2k-2s})(p^{2k}-1)}. \label{rewrite_b_chi_0_odd}
\end{align}
It suffices to check that (\ref{rewrite_b_chi_0_even}) and (\ref{rewrite_b_chi_0_odd}) satisfies the inductive formula (\ref{inductive_b_psi}). For (\ref{rewrite_b_chi_0_even}) we have to show
\begin{align}
& b_{\chi_0}(s)_{i,i+2r}  = \frac{p^{-s(i+2r) + (i+2r)(i+2r+1)/2}}{p^{r(2i+2r+1)-2sr}-1} \label{compute_b_chi_0-1}\\ 
& \quad \times \sum_{t=0}^{r-1} \bigl( m_{\chi_0}(s)_{i+2t,i+2r} \, b_{\chi_0}(s)_{i,i+2t} +  m_{\chi_0}(s)_{i+2t+1,i+2r} \, b_{\chi_0}(s)_{i,i+2t+1} \bigr). \notag
\end{align}
In the right hand side, we have
\[ 
m_{\chi_0}(s)_{i+2t,i+2r} \, b_{\chi_0}(s)_{i,i+2t} \\   = p^{\la}(1-p^{i+2t+1-s}) R(t) \]
with $\la = si-(i+2r)(i+2r+1)/2 + (r-t)(r-t+1)$,
\[ R(t)= \frac{ \displaystyle (-1)^t  \prod_{k=i+1}^{i+2r}(p^k-1)}{\displaystyle (1-p^{i+1-s}) \Biggl( \prod_{k=1}^t (1-p^{2i+1+2k-s})(p^{2k}-1) \Biggr) \prod_{k=1}^{r-t}(p^k-1)}. \]
Similarly we have
\[ m_{\chi_0}(s)_{i+2t+1,i+2r} \, b_{\chi_0}(s)_{i,i+2t+1}   = -p^{\la -2r+2t}(p^{2r-2t}-1) R(t). \]
Since $(1-p^{i+2t+1-s})-p^{-2r+2t}(p^{2r-2t}-1) = p^{-2r+2t}(1-p^{i+2r+1-s})$, the right hand side of (\ref{compute_b_chi_0-1}) is given by
\[ \frac{p^{-2sr}}{p^{r(2i+2r+1)-2sr}-1} \sum_{t=0}^{r-1}  (-1)^t p^{(r-t)(r-t-1)}(1-p^{i+2r+1-s}) R(t).\]
We can rewrite it as
\[ \frac{1}{p^{r(2i+2r+1)-2sr}-1} \cdot 
\frac{\displaystyle (-1)^r p^{-2rs} (1-p^{i+2r+1-s}) \prod_{k=i+1}^{i+2r} (p^k-1)}{\displaystyle (1-p^{i+1-s}) \prod_{k=1}^r (1-p^{2i+1+2k-2s})(p^{2k}-1)} \times C
\]
with
\[ C = \sum_{t=0}^{r-1} \frac{\displaystyle (-1)^{r-t}p^{(r-t)(r-t-1)}  \prod_{k=t+1}^{r} (1-p^{2i+1+2k-2s})(p^{2k}-1)}{\displaystyle \prod_{k=1}^{r-t}(p^{2k}-1)}. \]
Change $t \mapsto r-t$ we have
\[ C = \sum_{t=0}^{r-1} \frac{\displaystyle (-1)^{t}p^{t^2-t}  \prod_{k=r-t+1}^{r} (1-p^{2i+1+2k-2s})(p^{2k}-1)}{\displaystyle \prod_{k=1}^{t}(p^{2k}-1)}. \]
In the product in the numerator $\prod_{k=r-t+1}^{r}$ we change  $k \mapsto r+1-k$, then we have
\begin{align*}
C & = \sum_{t=0}^{r-1} \frac{\displaystyle (-1)^{t}p^{t^2-t}  \prod_{k=1}^{t} (1-p^{2i+2r+3-2k-2s})(p^{2r+2-2k}-1)}{\displaystyle \prod_{k=1}^{t}(p^{2k}-1)} \\
& =  \sum_{t=0}^{r-1} \prod_{k=1}^t \frac{(p^{2i+2r+1-2s}-p^{2k-2})(p^{2r+2-2k}-1)}{p^{2k}-1}.
\end{align*}
Now our assertion follows from Lemma \ref{lem_product_sum_1}, applying $x = p^2$ and $y = p^{2i+2r+1-2s}$. The proof of (\ref{rewrite_b_chi_0_odd}) is similar. Thus we have proved our assertion.

\end{proof}

On the other hand, $c_{\psi}(s)_{ij}$ is given as follows.

\begin{proposition} \label{prop_explicit_c_psi_ij}
\begin{enumerate}[$(1)$]
\item Let $\psi = \chi_0$. For $i \le j$ we have
\[ c_{\chi_0}(s)_{ij} = \begin{cases}  \frac{\displaystyle p^{\frac{(j-i)(j-i+2)}{4}-(j-i)s}(1-p^{i-s}) \prod_{k=i+1}^{j}(p^k-1)}{\displaystyle (1-p^{j-s}) \prod_{k=1}^{(j-i)/2} (1-p^{i+j-1+2k-2s})(p^{2k}-1)} & i \equiv j \bmod 2 \\
\displaystyle \frac{\displaystyle p^{\frac{(j-i)^2-1}{4}-(j-i)s} \prod_{k=i+1}^{j}(p^k-1)}{ \displaystyle (1-p^{j-s}) \prod_{k=1}^{(j-i-1)/2} (1-p^{i+j+2k-2s})(p^{2k}-1)} & i \not\equiv j \bmod 2 \end{cases} \]

\item Let $\psi = \chi_p$. For $i \le j$ and $i \equiv j \bmod 2$, we have
\[ c_{\chi_p}(s)_{ij} = \frac{\displaystyle \chi_p(-1)^{\frac{j-i}{2}} p^{\frac{j-i}{2}(\frac{j-i}{2}-2s)} \prod_{k=i+1}^{j} (p^k-1)}{\displaystyle \prod_{k=1}^{(j-i)/2} \bigl(1-p^{i+j-1+2k-2s} \bigr)(p^{2k}-1)}. \]
\end{enumerate}
\end{proposition}
In particular we have $c_\psi(s)_{ii} = 1$.

\begin{proof}
We only explain (2), since (1) can be proved similarly.
Write $j = i+2r$ with $r \ge 0$, then our assertion is rewritten
\[ c_{\chi_p}(s)_{i,i+2r} = \frac{\displaystyle \chi_p(-1)^r p^{r(r-2s)} \prod_{k=i+1}^{i+2r} (p^k-1)}{\displaystyle \prod_{k=1}^r \bigl(1-p^{2i+2r-1+2k-2s} \bigr)(p^{2k}-1)}. \]
We shall show the above formula by induction on $r$. 
Since the matrix $\bigl( c_{\chi_p}(s)_{ij} \bigr)_{0 \le i,j \le n}$ is the inverse matrix of $(b_{\chi_p}(s)_{ij})_{0, \le i,j \le n}$, we have
\begin{equation} c_{\chi_p}(s)_{i,i+2r} = - \sum_{t=1}^r b_{\chi_p}(s)_{i,i+2t} \, c_{\chi_p}(s)_{i+2t, i+2r}. \label{compute_c_psi-1}
\end{equation}
By the induction hypothesis and Proposition \ref{prop_explicit_b_psi}, we have
\begin{align*}
& b_{\chi_p}(s)_{i,i+2t} \, c_{\chi_p}(s)_{i+2t,i+2r}  = (-1)^t \chi_p(-1)^r p^{t(1-2s) + (r-t)(r-t-2s)} \\
&  \times \frac{\displaystyle  \prod_{k=i+1}^{i+2r}(p^k-1)}{\displaystyle  \Bigl( \prod_{k=1}^t (1-p^{2i+1+2k-2s})(p^{2k}-1) \Bigr) \Bigl( \displaystyle \prod_{k=1}^{r-t} (1-p^{2i+2t+2r-1+2k-2s})(p^{2k}-1) \Bigr) }. 
\end{align*}
We rewrite the right-hand side of (\ref{compute_c_psi-1}) of the form
\[ \frac{\displaystyle \chi_p(-1)^r p^{r(r-2s)} \prod_{k=i+1}^{i+2r} (p^k-1)}{\displaystyle \prod_{k=1}^r \bigl(1-p^{2i+2r-1+2k-2s} \bigr)(p^{2k}-1)} \times R. \]
Note that
\[ \frac{\displaystyle \prod_{k=1}^r (p^{2k}-1)}{\displaystyle \prod_{k=1}^t(p^{2k}-1) \prod_{k=1}^{r-t}(p^{2k}-1)} = \prod_{k=1}^t \frac{ p^{2r-2t+2k} -1}{p^{2k}-1} = \prod_{k=1}^t \frac{p^{2r+2-2k}-1}{p^{2k}-1}, \]
here the last equality follows by changing $k \mapsto t-k+1$ in the numerator.
Moreover we can write
\[ \prod_{k=1}^{r-t}(1-p^{2i+2t+2r-1+2k-2s}) = \prod_{k=t+1}^{r} (1-p^{2i+2r-1 + 2k-2s}). \]
Thus we have
\begin{align*} R & = - \sum_{t=1}^{r} (-1)^t p^{t^2+t -2tr} \prod_{k=1}^t \frac{ (1-p^{2i+2r-1+2k-2s})(p^{2r+2-2k}-1)}{(1-p^{2i+1+2k-2s})(p^{2k}-1)} \\
& = - \sum_{t=1}^{r} \prod_{k=1}^t \frac{ (p^{2i+2r-1+2k-2s}-1)(p^{2r+2-2k}-1)}{(p^{2r-2k} -p^{2i+2r+1-2s})(p^{2k}-1)}.
\end{align*}
By Lemma \ref{lem_product_sum_2} we have $R=1$ and get our result.
\end{proof}

\begin{definition} We denote the ramified Siegel series of $E^{n,(\nu)}_{k,\psi}(Z,s)$ for the half integral matrix $N \in \Sym^n(\Z)^*$ by $S^{(\nu)}_n(\psi,N,s)$ and call it the $U(p)$-characteristic ramified Siegel series.
\end{definition}

In other words we can write
\begin{equation} \label{ramified_Hecke_eigen_standard_base}
 S^{(\nu)}_n(\psi,N,s) = \sum_{ r=\nu  }^n b_{\psi}(s)_{\nu r} \, S^{w_r}_n(\psi,N,s). 
\end{equation}
Then
\begin{equation} \label{ramified_standard_base_Hecke_eigen}
 S^{w_\nu}_n(\psi,N,s) = \sum_{r = \nu}^n c_{\psi}(s)_{\nu r} \, S^{(r)}_{n}(\psi,N,s)
\end{equation}
holds. As a consequence, our aim is to find an explicit formula of $S^{(\nu)}_{n}(\psi,N,s)$. We note that 
\[ S^{w_n}_n(\psi,N,s) = S^{(n)}_{n}(\psi,N,s) = 1. \]

\section{Functional equations of ramified Siegel series} \label{Sec_functional_equation}

In this section we show the functional equations of the $U(p)$-characteristic ramified Siegel series $S^{(\nu)}_{n}(\psi,N,s)$ for non-degenerate $N \in \Sym^n(\Z)^*$. 

\subsection{Fourier coefficient of Siegel Eisenstein series}

First we recall the functional equations of Siegel Eisenstein series of level $p$. Let 
\[ \ve_p = \begin{cases} 1  & p \equiv 1 \bmod 4, \\ \sqrt{-1} & p \equiv 3 \bmod 4, \end{cases} \quad \delta_p = \begin{cases} 0 & p \equiv 1 \bmod 4, \\ 1 & p \equiv 3 \bmod 4. \end{cases} \]
We put for $0 \le \nu \le n$,
\begin{align*}
  \g_\nu (\chi_0, s)  & =  \frac{1-p^{-s}}{1-p^{\nu-s}} \prod_{i=1}^{[\nu/2]} \frac{1-p^{2i-2s}}{1-p^{2\nu+1 -2i-2s}}, \\
  \g_\nu (\chi_p, s)  & =  (\ve_p p^{-1/2})^\nu  \prod_{i=1}^{[\nu/2]} \frac{1-p^{2i-2s}}{1-p^{2\nu+1 -2i-2s}}
\end{align*}
and set
\begin{align*} 
\mathbb{E}^{n,(\nu)}_{k,\chi_0}(Z,s) & = \gamma_\nu(\chi_0,s ) \frac{\Gamma_n\Bigl(\dfrac{s+k}{2} \Bigr)}{\Gamma_n\Bigl( \dfrac{s}{2} \Bigr)} \xi(s) \prod_{j=1}^{[n/2]} \xi(2s-2j) E^{n,(\nu)}_{k,\chi_0}(Z,s), \\
\mathbb{E}^{n,(\nu)}_{k,\chi_p}(Z,s) & = \gamma_\nu(\chi_p,s ) \frac{\Gamma_n\Bigl(\dfrac{s+k}{2} \Bigr)}{\Gamma_n\Bigl( \dfrac{s+\delta_p}{2} \Bigr)} \xi(\chi_p,s) \prod_{j=1}^{[n/2]} \xi(2s-2j) E^{n,(\nu)}_{k,\chi_p}(Z,s).
\end{align*}
Here $\xi(s)$ (resp.\  $\xi(\chi_p,s)$) is the completed Riemann zeta function (resp.\ completed Dirichlet $L$-function), defined by
\[ \xi(s) = \pi^{-s/2} \Gamma \Bigl( \frac{s}{2} \Bigr) \zeta(s), \quad \xi(\chi_p,s) = \Bigl( \frac{p}{\pi} \Bigr)^{(s+\delta_p)/2} \Gamma \Bigl( \frac{s + \delta_p}{2} \Bigr) L(\chi_p,s), \]
that satisfies $\xi(1-s) = \xi(s)$ and $\xi(\chi_p,1-s) = \xi(\chi_p,s)$ respectively.

\begin{theorem}[{\cite[Theorem 4.3, Theorem 4.8]{Gu4}}] \label{thm_FE_Siegel_Eisenstein_level_p}
For $\psi = \chi_0$ or $\chi_p$, we have the functional equation
\[ \mathbb{E}^{n,(\nu)}_{k,\psi}(Z,s) = \mathbb{E}^{n,(n-\nu)}_{k,\psi}(Z,n+1-s) . \]
\end{theorem}

One can induce the functional equations of ramified Siegel series between $S^{(\nu)}_n(\psi,N,s)$ and $S^{(n-\nu)}_n(\psi,N,n+1-s)$ by this theorem. We refer the calculation in \cite{BK}, however the proof is simpler than that. This is because we can get  our assertion directly from Theorem \ref{thm_FE_Siegel_Eisenstein_level_p}, for we already know the functional equations of ordinary Siegel series.

Take a non-degenerate $N \in \Sym^n(\Z)^*$,  put $D_N = (-4)^{[n/2]} \det N$.  For each prime number $q$,  we put $d_q(N) = \ord_q D_N$.  The Hasse invariant $h_q(N)$ of $N$ is defined as follows; if $N$ is $\GL_n(\Q_p)$-equivalent to $\diag(a_1, \ldots, a_n)$ then we define $h_q(N) = \prod_{i \le j}(a_i,a_j)_q$, here $( \cdot \, , \cdot )_q$ stands for the Hilbert symbol. We set
\[ \zeta_q(N) = \begin{cases} 1 & \text{if $n$ is even,} \\ h_q(N) \, \bigl( \det N, (-1)^{\frac{n-1}{2}} \det N \bigr)_q \, \bigl(-1,-1 \bigr)_q^{\frac{n^2-1}{8}} & \text{if $n$ is odd.}\end{cases} \]

In the case that $n$ is even, let $\chi_N$ be the Dirichlet character associated to the quadratic extension $\Q(\sqrt{D_N})/\Q$, that is given concretely as follows. We denote the discriminant of $\Q(\sqrt{D_N})/\Q$ by $\fd_N$. Then the conductor of $\chi_N$ is $|\fd_N|$, if $(q,\fd_N) = 1$ for an odd prime $q$ then $\chi_N(q) = (\frac{\fd_N}{q})$; $\chi_N(-1) = \sgn( \fd_N)$ and finally
\[ \chi_N(2) = \begin{cases} 1 & \fd_N \equiv \pm 1 \bmod 8 \\ -1 & \fd_N \equiv \pm 3 \bmod 8 \\ 0 & \text{$\fd_N$ is even.} \end{cases} \]
We set
\[ e_q(N) = \begin{cases} d_q(N)-\ord_q \fd_N & \text{if $n$ is even}, \\ d_q(N) & \text{if $n$ is odd.} \end{cases} \]

We need further notations for $q=p$ case. We write $D_N = p^{d_p(N)} D_N'$ and put
\[ \eta_p(N) = \begin{cases} 1 & \text{if $n$ is even,} \\ \chi_p(D_N') \zeta_p(N) & \text{if $n$ is odd.} \end{cases} \]
When $n$ is even, let $\chi_N^*$ be the primitive character arising from $\chi_N \chi_p$. If we put
\[ \fd_N^* = \begin{cases} p \fd_N & \text{if  $(p, \fd_N ) = 1$,} \\ p^{-1} \fd_N & \text{if $p \mid \fd_N$}, \end{cases} \]
then the conductor of $\chi_N^*$ is $| \fd_N^*|$. We have $\chi_N^*(p) = 0$ or $\chi_N^*(p)= \chi_p(-\fd_N^*)$ according as $(p, \fd_N) = 1$ or $p \mid \fd_N$, respectively.
Also in the case of  $n$ is odd,  we understand that $\chi_N(p) = \chi_N^*(p) = 0$.

Now we state the main result of this section.
For $0 \le \nu \le n$, we define the \emph{zeta factor} $\beta_n^{(\nu)}(\psi,N,s)$ of the $U(p)$-characteristic ramified Siegel series $S^{(\nu)}_n(\chi_p,N,s)$ by
\begin{align*}
 \beta_n^{(\nu)}(\chi_0,N,s) & = \frac{(1-p^{\nu-s})\displaystyle \prod_{i=1}^{[\nu/2]} (1- p^{2 \nu + 1- 2i-2s}) \prod_{i=[\nu/2]+1}^{[n/2]}(1-p^{2i-2s})}{(1- \chi_{N}(p) p^{n/2-s})}, \\[5pt]
 \beta_n^{(\nu)}(\chi_p,N,s) & = \frac{\displaystyle \prod_{i=1}^{[\nu/2]} (1- p^{2 \nu + 1- 2i-2s}) \prod_{i=[\nu/2]+1}^{[n/2]}(1-p^{2i-2s})}{(\ve_p p^{-1/2})^\nu (1- \chi_{N}^*(p) p^{n/2-s})},
\end{align*}
and define the \emph{principal part} $F_n^{(\nu)}(\psi,N,s)$ of $S^{(\nu)}_n(\psi,N,s)$ by
\[ S_n^{(\nu)}(\psi,N,s) = \beta_n^{(\nu)}(\psi,N,s) F_n^{(\nu)}(\psi,N,s). \]

\begin{theorem} \label{thm_FE_ramified_Siegel_series}
Let $N \in \Sym^n(\Z)^*$ be non-degenerate.  If $n$ is odd then we put $a_N=1$; if $n$ is even then we take $a_N \in \{0, 1 \}$ so that $a_N \equiv d_p(N) \bmod 2$. Then the principal part of the ramified Siegel series satisfies the following functional equations. If $\psi = \chi_0$, then
\[ F_n^{(n-\nu)}(\chi_0,N,n+1-s)  = \zeta_p(N) p^{(s-\frac{n+1}{2})e_p(N)} F_n^{(\nu)}(\chi_0,N,s), \]
if $\psi = \chi_p$ then
\[ F_n^{(n-\nu)}(\chi_p,N,n+1-s)  = \eta_p(N) p^{(s-\frac{n+1}{2})(d_p(N)+a_N)} F_n^{(\nu)}(\chi_p,N,s).
\]
\end{theorem}

Since $S_n^{(n)}(\chi_p,N,s)=1$, we can find the explicit formula of $S_n^{(0)}(\psi,N,s)$ by using the functional equation.

\begin{corollary} \label{cor_explicit_formula_S_n^(n)}
Under the same notation as in Theorem \emph{\ref{thm_FE_ramified_Siegel_series}}, we have the following.

\begin{enumerate}[$(1)$]
\item For the trivial character case we have
\begin{align*}
 S_n^{(0)}(\chi_0,N,s)  & = \zeta_p(N) p^{(\frac{n+1}{2}-s)e_p(N)}  
\\ 
& \qquad  \times \frac{ \displaystyle (1-p^{-s}) \bigl( 1-\chi_N(p) p^{s-\frac{n}{2}-1} \bigr) \prod_{i=1}^{[n/2]} (1-p^{2i-2s}) }{\displaystyle (1-p^{s-1}) \bigl(1-\chi_N(p)p^{\frac{n}{2}-s} \bigr) \prod_{i=1}^{[n/2]}(1-p^{2s-2i-1})}.
\end{align*}

\item For the quadratic character case we have
\begin{align*} S_n^{(0)}(\chi_p,N,s) & = \eta_p(N) \, \ve_p^n  \, p^{ (\frac{n+1}{2}-s)(d_p(N)+a_N)-\frac{n}{2}}   \\
& \qquad \times \frac{ \displaystyle (1-\chi_N^*(p) p^{s-\frac{n}{2}-1}) \prod_{j=1}^{[n/2]} (1-p^{2j-2s})}{ \displaystyle (1-\chi_N^*(p)p^{\frac{n}{2}-s}) \prod_{j=1}^{[n/2]}(1-p^{2s-1-2j})}.
\end{align*}
\end{enumerate}
\end{corollary}

In the rest of this section we shall prove the theorem in the case of quadratic character. For the trivial character case, one can prove it similarly.
For the proof we follow the calculation in \cite{BK}.

Now assume that $\psi = \chi_p$. It is convenient to consider the Siegel Eisenstein series of weight $\delta_p$. We denote the Fourier coefficient of
\[ \mathbb{E}^{n,(\nu)}_{\delta_p,\chi_p}(Z,s) = \gamma_\nu(\chi_p,s) \xi(\chi_p,s) \prod_{j=1}^{[n/2]}\xi(2s-2j) E^{n,(\nu)}_{\delta_p,\chi_p}(Z,s) \]
at $N \in \Sym^n(\Z_p)^*$ by $\mathbb{A}_N^{(\nu)}(Y,s)$ or simply $\A_N^{(\nu)}(s)$. If $N$ is non-degenerate, by Proposition \ref{prop_Fourier_exp_Eisenstein_level_p_1} we have
\begin{align} \A_N^{(\nu)}(Y,s) & =  \det(Y)^{(s-\delta_p)/2} \g_\nu(N,s) \xi(\chi_p,s) \prod_{j=1}^{[n/2]} \xi(2s-2j)  \label{A_N^(nu)_1}  \\
&  \times \Xi_n \Bigl(Y,N, \frac{s+\delta_p}{2},\frac{s-\delta_p}{2} \Bigr) S_n^{(\nu)}(\chi_p, N,p^{-s}) \prod_{q \ne p} S_{n,q}(N, \chi_p(q)q^{-s}). \notag
\end{align}

First we study the confluent hypergeometric function $\Xi_n(Y,N, \alpha, \beta)$. In the following, we denote by $C$ a constant that does not depend on $s$.
Assume that the signature of $N$ is $(a+,b-)$. By \cite[(4.6.K), (4.30.K), (4.34.K)]{Shi1}, we can write
\begin{align*} & \qquad (\det Y)^{(s-\delta_p)/2} \Xi_n \Bigl( Y,N; \frac{s+\delta_p}{2}, \frac{s-\delta_p}{2}\Bigr) \\
& = C \, \Gamma_a \Bigl ( \frac{s+\delta_p}{2} \Bigr)^{-1} \Gamma_b \Bigl( \frac{s-\delta_p}{2} \Bigr)^{-1} |\det \pi N|^{s/2} \psi \Bigl(2 \pi Y, N; \frac{s+\delta_p}{2}, \frac{s-\delta_p}{2} \Bigr),
\end{align*}
here the function  $\widetilde{\psi}(s) = \psi(2 \pi Y,N, \frac{s+\delta_p}{2}, \frac{s-\delta_p}{2} )$ satisfies $\widetilde{\psi}(n+1-s) = \widetilde{\psi}(s)$ by \cite[(4.31)]{Shi1}.

Next we recall the functional equation of the ordinary Siegel series.  It is known that if $n$ is even, then
\[  S_{n,q}(N,\chi_p(q)q^{-s}) =  \frac{(1-\chi_p(q)q^{-s}) \prod_{j=1}^{n/2} (1-q^{2j-2s}) }{(1- \chi_N \chi_p(q) q^{n/2-s})}   F_{n,q}(N, \chi_p(q) q^{-s}), \]
and if $n$ is odd then
\[ S_{n,q}(N, \chi_p(q)q^{-s}) =  (1-\chi_p(q)q^{-s}) \prod_{j=1}^{(n-1)/2}(1-q^{2j-2s}) F_{n,q}(N, \chi_p(q) q^{-s}), \]
with a certain polynomial  $F_{n,q}(N,X)$ in $X$, that coincides to $1$ for almost all $q$ (cf.\ \cite[Proposition 3.6]{Shi2}). The term $F_{n,q}(N,\chi_p(q)q^{-s})$ is called the \emph{polynomial part} of the Siegel series.

\begin{theorem}[{\cite[Theorem 3.2]{Kat}, \cite[Theorem]{BK}}]\label{thm_FE_unram_Siegel_series}
For  $q \ne p$, the following functional equation holds.
\[ F_{n,q}(N, \chi_p(q)q^{s-n-1}) = \zeta_q(N) \, \bigl(\chi_p(q) q^{s-\frac{n+1}{2}} \bigr)^{e_q(N)} F_{n,q}(N, \chi_p(q)q^{-s}). \]
\end{theorem}

In (\ref{A_N^(nu)_1}) we have
\begin{align*} & \g_n(N,s) S^{(\nu)}_n(\chi_p,N,p^{-s}) \prod_{q \ne n} S_{n,q}(N, \chi_p(q)q^{-s}) \\
& = \frac{L(\chi_N^*, s-n/2)}{L(\chi_p,s) \prod_{i=1}^{[n/2]}\zeta(2s-2i)} F^{(\nu)}_n(\chi_p,N,s) \prod_{q \ne p} F_{n,q}(N,\chi_p(q)q^{-s}),
\end{align*}
here we understand $L(\chi_N^*,s-n/2) = 1$ is $n$ is odd.

Set
$\Gamma^*_r(s) = \pi^{-r(r-1)/4} \Gamma_r(s)= {\prod_{j=0}^{r-1} \Gamma (s- j/2)}$,
then by Legendre's duplication formula one can write 
\[ \prod_{j=1}^r \Gamma(s-j) = \pi^{-r/2} 2^{sr - r(r+3)/2}\Gamma^*_{r}\Bigl(\frac{s-1}{2} \Bigr) \Gamma^*_r \Bigl(\frac{s}{2} \Bigr). \]
Using it we can rewrite (\ref{A_N^(nu)_1}) as follows. If $n$ is odd then
\begin{align}
 \A_N^{(\nu)}(s) & =  C \widetilde{\psi}(s) |pD_N|^{s/2}  \Gamma \Bigl( \frac{s+\delta_p}{2} \Bigr) \frac{\Gamma^*_{\frac{n-1}{2}} \Bigl( \dfrac{s-1}{2} \Bigr) \Gamma_{\frac{n-1}{2}}^* \Bigl( \dfrac{s}{2} \Bigr)}{\Gamma_a^* \Bigl(\dfrac{s+\delta_p}{2} \Bigr) \Gamma^*_b \Bigl( \dfrac{s-\delta_p}{2} \Bigr)} \,  \label{A_N^(nu)_2_for_odd_n} \\
& \quad \times   F_n^{(\nu)}(\chi_p,N,s) \prod_{q \ne p} F_{n,q}(N,\chi_p(q)q^{-s}), \notag
\end{align}
and if $n$ is even then
\begin{align}
 \A_N^{(\nu)}(s)  & =  C \widetilde{\psi}(s) |pD_N|^{s/2} \Gamma \Bigl( \frac{s+\delta_p}{2} \Bigr) \frac{\Gamma^*_{\frac{n}{2}} \Bigl( \dfrac{s-1}{2} \Bigr) \Gamma_{\frac{n}{2}}^* \Bigl( \dfrac{s}{2} \Bigr)}{\Gamma_a^* \Bigl(\dfrac{s+\delta_p}{2} \Bigr) \Gamma^*_b \Bigl( \dfrac{s-\delta_p}{2} \Bigr)} \label{A_N^(nu)_2_for_even_n} \\
& \quad \times  \pi^{-s/2}  L(\chi_N^*,s-n/2) F_n^{(\nu)}(\chi_p,N,s) \prod_{q \ne p} F_{n,q}(N,\chi_p(q)q^{-s}). \notag
\end{align}

\subsection{Odd degree case}

First we assume that $n$ is odd. We have
\[
 \A_N^{(\nu)}(s)   =  C \widetilde{\psi}(s) |pD_N|^{s/2} g(s) \,  F_n^{(\nu)}(\chi_p,N,s) \prod_{q \ne p} F_{n,q}(N,\chi_p(q)q^{-s}),
\]
with
\[ g(s) = \begin{cases} \frac{\displaystyle \Gamma^*_{\frac{n+1}{2}} \Bigl( \dfrac{s}{2} \Bigr) \Gamma_{\frac{n-1}{2}}^* \Bigl( \dfrac{s}{2} \Bigr) \vrule width0pt depth9pt}{\displaystyle \Gamma_a^* \Bigl(\dfrac{s}{2} \Bigr) \Gamma^*_{n-a} \Bigl( \dfrac{s}{2} \Bigr)\vrule width0pt height15pt} & \text{if $p \equiv 1 \bmod 4$,} \\[22pt]
 \frac{\displaystyle \Gamma^*_{\frac{n+1}{2}} \Bigl( \dfrac{s+1}{2} \Bigr) \Gamma_{\frac{n-1}{2}}^* \Bigl( \dfrac{s-1}{2} \Bigr)\vrule width0pt depth9pt}{\displaystyle \Gamma_a^* \Bigl(\dfrac{s+1}{2} \Bigr) \Gamma^*_{n-a} \Bigl( \dfrac{s-1}{2} \Bigr)\vrule width0pt height15pt} & \text{if $p \equiv 3 \bmod 4$.} \end{cases} \]

\begin{lemma} \label{lem_FE_g(s)}
We have the following equations.
\begin{enumerate}[$(1)$]
\item If $p \equiv 1 \bmod 4$ then $g(n+1-s) = (-1)^{\frac{(a-b)^2-1}{8}} g(s)$.

\item If $p \equiv 3 \bmod 4$ then $g(n+1-s) = (-1)^{\frac{(a-b-2)^2-1}{8}} g(s)$.
\end{enumerate}
\end{lemma}

At least the case for $p \equiv 1 \bmod 4$, this Lemma is already proved essentially in \cite[p.284]{BK}, however here we prove it using the following lemma due to Mizumoto.

\begin{lemma}[{\cite[Lemma 6.1]{Miz}}] \label{lem_Mizumoto_Gamma_FE}
 For $r \in \Z$ we have
\[ \frac{\Gamma_m(s)}{\Gamma_m(s+r)} = (-1)^{mr} \frac{ \Gamma_m \Bigl( \dfrac{m+1}{2}-r-s \Bigr)}{\Gamma_m \Bigl(\dfrac{m+1}{2}-s \Bigr)}. \]

\end{lemma}

\begin{proof}[{Proof of Lemma \emph{\ref{lem_FE_g(s)}}}] First note that for $n \ge m$ we have
\[ \dfrac{\Gamma^*_n(s)}{\Gamma^*_m(s)} = \Gamma^*_{n-m}\Bigl( s - \dfrac{m}{2} \Bigr). \]
 
\noindent
(1) \  It suffices to consider the case $a \le (n-1)/2$, since it is symmetric for $a$ and $n-a$. By the above remark we can rewrite $g(s)$ as
\begin{equation} \tag{$*$} 
g(s)= \frac{\Gamma^*_{\frac{n+1}{2}-a}\Bigl( \dfrac{s}{2}-\dfrac{a}{2} \Bigr)}{\Gamma^*_{\frac{n+1}{2}-a} \Bigl(\dfrac{s}{2} - \dfrac{n-1}{4} \Bigr)}
\end{equation}
or
\begin{equation} \tag{$**$}
g(s) = \dfrac{\Gamma^*_{\frac{n-1}{2}-a}\Bigl( \dfrac{s}{2}-\dfrac{a}{2} \Bigr)}{\Gamma^*_{\frac{n-1}{2}-a} \Bigl(\dfrac{s}{2} - \dfrac{n+1}{4} \Bigr)}
\end{equation}
Note that either $\dfrac{a}{2}-\dfrac{n-1}{4}$ or $\dfrac{a}{2}-\dfrac{n+1}{4}$ is an integer. In the former case we apply Lemma \ref{lem_Mizumoto_Gamma_FE} to $(*)$ and get
\begin{align*}  g(s)&  = (-1)^{(\frac{n+1}{2}-a)(\frac{a}{2}-\frac{n-1}{4})} \frac{\Gamma^*_{\frac{n-1}{2}-a} \Bigl( \dfrac{n+3}{4}-\dfrac{a}{2} - \dfrac{s}{2} + \dfrac{n-1}{4} \Bigl)}{\Gamma^*_{\frac{n+1}{2}-a} \Bigl( \dfrac{n+3}{4}-\dfrac{a}{2} - \dfrac{s}{2} + \dfrac{a}{2} \Bigr)}\\
& =(-1)^{(1-(a-b)^2)/8} \frac{\Gamma^*_{\frac{n-1}{2}-a} \Bigl( \dfrac{n+1-s}{2} - \dfrac{a}{2} \Bigr)}{\Gamma^*_{\frac{n-1}{2}-a} \Bigl( \dfrac{n+1-s}{2}- \dfrac{n-1}{4} \Bigr)},
\end{align*}
in the latter case we apply Lemma \ref{lem_Mizumoto_Gamma_FE} to $(**)$ to get
\begin{align*}  g(s)&  = (-1)^{(\frac{n-1}{2}-a)(\frac{a}{2}-\frac{n+1}{4})} \frac{\Gamma^*_{\frac{n-1}{2}-a} \Bigl( \dfrac{n+1}{4}-\dfrac{a}{2} - \dfrac{s}{2} + \dfrac{n+1}{4} \Bigl)}{\Gamma^*_{\frac{n-1}{2}-a} \Bigl( \dfrac{n+1}{4}-\dfrac{a}{2} - \dfrac{s}{2} + \dfrac{a}{2} \Bigr)}\\
& =(-1)^{(1-(a-b)^2)/8} \frac{\Gamma^*_{\frac{n-1}{2}-a} \Bigl( \dfrac{n+1-s}{2} - \dfrac{a}{2} \Bigr)}{\Gamma^*_{\frac{n-1}{2}-a} \Bigl( \dfrac{n+1-s}{2}- \dfrac{n+1}{4} \Bigr)}.
\end{align*}
This proves our assertion.

\bigskip
\noindent
(2) Assume $p \equiv 3 \bmod 4$. Multiplying $\Gamma\Bigl( \dfrac{s+1}{2} \Bigr) \Gamma\Bigl( \dfrac{s}{2} \Bigr)$ for both the numerator and the denominator, we can write
\[ g(s) = \frac{\Gamma_{\frac{n+3}{2}}^* \Bigl( \dfrac{s+1}{2} \Bigr) \Gamma_{\frac{n+1}{2}}^* \Bigl( \dfrac{s+1}{2} \Bigr)}{\Gamma_{a}^* \Bigl( \dfrac{s+1}{2} \Bigr) \Gamma_{n+2-a}^* \Bigl( \dfrac{s+1}{2} \Bigr)}. \]
Thus if $a \le (n+1)/2$ we can write
\[ g(s) = \dfrac{ \Gamma^*_{\frac{n+1}{2}-a}\Bigl( \dfrac{s+1}{2} -\dfrac{a}{2} \Bigr)}{\Gamma^*_{\frac{n+1}{2}-a} \Bigl( \dfrac{s+1}{2} - \dfrac{n+3}{4} \Bigl)} = \dfrac{\Gamma^*_{\frac{n+3}{2}-a}\Bigl(\dfrac{s+1}{2}-\dfrac{a}{2} \Bigl)}{\Gamma^*_{\frac{n+3}{2}-a} \Bigl( \dfrac{s+1}{2}-\dfrac{n+1}{4} \Bigl)}. \]
Since either $\dfrac{a}{2}-\dfrac{n+3}{4}$ or $\dfrac{a}{2}-\dfrac{n+1}{4}$ is an integer, we can apply Lemma \ref{lem_Mizumoto_Gamma_FE} and get our assertion.
The case for $a \ge (n+3)/2$ is similar.
\end{proof}

We set
\[ \zeta_\infty(N) = h_\infty(N) \bigl(\det N, (-1)^{\frac{n-1}{2}} \det N \bigr)_\infty (-1,-1)^{\frac{n^2-1}{8}}_\infty, \]
then $\prod_{q \le \infty} \zeta_q(N) = 1$ holds. It is easy to show that $\zeta_\infty(N) = (-1)^{\frac{(a-b)^2-1}{8}}$, 
thus we can write Lemma \ref{lem_FE_g(s)} as $g(n+1-s) = \chi_p(-1)^{\frac{n-1}{2}+b} \zeta_\infty(N) g(s)$.

In the description
\[ \mathbb{A}^{(\nu)}_N(s) = C \widetilde{\psi}(s)g(s) \, |p D_N|^{s/2}  F_n^{(\nu)}(\chi_p,N,s) \prod_{q \ne p} F_{n,q}(N,\chi_p(q)q^{-s}),\]
we apply Theorem \ref{thm_FE_unram_Siegel_series} to $\prod_{q \ne p} F_{n,q}(N,\chi_p(q)q^{-s})$. Recall that $D_N = p^{d_p(N)} D'_N$, thus we have $\prod_{q \ne p} q^{e_q(N)} = |D'_N|$. By Theorem \ref{thm_FE_Siegel_Eisenstein_level_p} $\A^{(n-\nu)}_N(n+1-s) = \A^{(\nu)}_N(s)$ holds, therefore we have
\begin{align*}  & |p D_N|^{(n+1-s)/2} F_n^{(n-\nu)}(\chi_p,N,n+1-s) \, \chi_p(|D'_N|) \,| D'_N|^{-\frac{n+1}{2}+s} \chi_p(-1)^{\frac{n-1}{2}+b} \\
& = |p D_N|^{s/2} F_n^{(\nu)}(\chi_p,N,s).
\end{align*}
Since $D'_N = (-1)^{\frac{n-1}{2} + b} |D'_N|$, we have proved Theorem \ref{thm_FE_ramified_Siegel_series} for an odd $n$. 

\subsection{Even degree case}

For the case that $n$ is even,  we have to rewrite $L(\chi_N^*,s-n/2)$ in (\ref{A_N^(nu)_2_for_even_n}) by using the completed $L$-function $\xi(\chi_N^*,s-n/2)$.

First assume that $p \equiv 1 \bmod 4$. Then $\chi_N^*(-1) = (-1)^{n/2-a}$, thus
\begin{align*} &  \xi(\chi_N^*, s- n/2)  \\
& = \begin{cases} \left( \dfrac{|\fd_N^*|}{\pi} \right)^{\frac{s}{2}-\frac{n}{4}} \Gamma \Bigl( \dfrac{s}{2}-\dfrac{n}{4} \Bigr) L(\chi_N^*, s-n/2) & \text{$n/2-a$ is even,} \vrule width0pt depth13pt \\
\left( \dfrac{|\fd_N^*|}{\pi} \right)^{\frac{s}{2}-\frac{n}{4}+\frac{1}{2}} \Gamma \Bigl( \dfrac{s}{2}-\dfrac{n}{4} + \dfrac{1}{2} \Bigr) L(\chi_N^*, s-n/2) & \text{$n/2-a$ is odd.} \end{cases} 
\end{align*}
Then we can write
\begin{align*} 
\A^{(\nu)}_N(s) & = C  \widetilde{\psi}(s) \, g_1(s)  \left| \frac{p D_N }{\fd_N^*} \right|^{s/2} \xi(\chi_N^*, s- n/2 ) \\ 
& \qquad \times 
F_n^{(\nu)}(\chi_p,N,s) \prod_{q \ne p} F_{n,q}(\chi_p,N, s), 
\end{align*}
with
\[ g_1(s) = \frac{ \Gamma^*_{\frac{n}{2}} \Bigl( \dfrac{s}{2} \Bigr) \Gamma^*_{\frac{n}{2}} \Bigl( \dfrac{s}{2} \Bigr)}{\Gamma_a^* \Bigl( \dfrac{s}{2} \Bigr) \Gamma^*_{n-a} \Bigl( \dfrac{s}{2} \Bigr)} \quad \text{or} \quad
\frac{ \Gamma^*_{\frac{n}{2}+1} \Bigl( \dfrac{s}{2} \Bigr) \Gamma^*_{\frac{n}{2}-1} \Bigl( \dfrac{s}{2} \Bigr)}{\Gamma_a^* \Bigl( \dfrac{s}{2} \Bigr) \Gamma^*_{n-a} \Bigl( \dfrac{s}{2} \Bigr)} \]
according as $n/2-a$ is even or odd.

Next assume that $p \equiv 3 \bmod 4$. Then similarly we have
\begin{align*} 
\A^{(\nu)}_N(s) & = C  \widetilde{\psi}(s) \, g_2(s)  \left| \frac{p D_N }{\fd_N^*} \right|^{s/2} \xi(\chi_N^*, s- n/2 ) \\
& \qquad \times  F_n^{(\nu)}(\chi_p,N,s) \prod_{q \ne p} F_{n,q}(\chi_p,N, s), 
\end{align*}
with
\[ g_2(s) = \frac{ \Gamma^*_{\frac{n}{2}} \Bigl( \dfrac{s+1}{2} \Bigr) \Gamma^*_{\frac{n}{2}} \Bigl( \dfrac{s-1}{2} \Bigr)}{\Gamma_a^* \Bigl( \dfrac{s+1}{2} \Bigr) \Gamma^*_{n-a} \Bigl( \dfrac{s-1}{2} \Bigr)} \quad \text{or} \quad
\frac{ \Gamma^*_{\frac{n}{2}+1} \Bigl( \dfrac{s+1}{2} \Bigr) \Gamma^*_{\frac{n}{2}-1} \Bigl( \dfrac{s-1}{2} \Bigr)}{\Gamma_a^* \Bigl( \dfrac{s+1}{2} \Bigr) \Gamma^*_{n-a} \Bigl( \dfrac{s-1}{2} \Bigr)} \]
according as $n/2-a$ is even or odd.

Similar to Lemma \ref{lem_FE_g(s)} one can prove the following.

\begin{lemma} \label{lem_FE_g_1_and_g_2}
We have $g_1(n+1-s) = g_1(s)$ and $g_2(n+1-s) = g_2(s)$.
\end{lemma}

We set $\delta_N^* = 0$ or $1$ according as $\chi_N^*(-1) = 1$ or $-1$ respectively, then the Gauss sum of $\chi_N^*$ can be calculated as $G(\chi_N^*) = (\sqrt{-1})^{\delta_N^*} \sqrt{|\fd^*_N|}$, thus 
the functional equation $\xi(\chi_N^*,s) = \xi(\chi_N^*,1-s)$ holds. Finally note that $e_q(N)$ for $q \ne p$ is even by definition, thus $\chi_p(q)^{e_q(N)} = 1$. Using $\A_N^{(n-\nu)}(n+1-s) = \A_N^{(\nu)}(s)$, one can prove Theorem \ref{thm_FE_ramified_Siegel_series} for even $n$, since we have
\[ \fd_N^* = \begin{cases}  p \fd_N & \text{$d_p(N)$ is even,} \\ p^{-1} \fd_N & \text{$d_p(N)$ is odd.} \end{cases} \]
Now our proof is finished.

\section{Recursion formulas of ramified Siegel series} \label{Sec_recursion}

Our aim is to give explicit formulas of $S^{w_\nu}_n(\psi_p,N,s)$ for any $0 \le \nu \le n$. By (\ref{ramified_standard_base_Hecke_eigen}) it suffices to find the formulas for $S^{(\nu)}_n(\psi,N,s)$. In order to find the recursion formula of $S^{(\nu)}_n(\psi,N,s)$,  we use the inductive relations of ramified Siegel series analogous to (\ref{inductive_relation_Katsurada_2}).

Let $\alpha_i \in \Z_p^\times$ for $1 \le i \le n$ and $0 \le u_1 \le u_2 \le \cdots \le u_n$ be integers. We put $N_0 = \diag(\alpha_i p^{u_1}, \ldots, \alpha_{n-1} p^{u_{n-1}}) \in \Sym^{n-1}(\Z_p)^*$. We also  define the elements in $\Sym^n(\Z_p)^*$ by $N = N_0 \perp \langle \alpha_n p^{u_n} \rangle$ and $N' = N_0 \perp \langle \alpha_n p^{u_n+2} \rangle$.

\subsection{The case of trivial character}
First we treat the case $\psi = \chi_0$. In this case the inductive relation of the characteristic ramified Siegel series comes from the ordinary one, that was essentially given by Katsurada.

\begin{theorem} \label{thm_inductive_relation_chi_0}
We define $H_n(N,s)$ as follows. If $n$ is even we put
\[ H_n(N,s) = \zeta_p(N_0) p^{\frac{1}{2} d_p(N_0) + (\frac{n+1}{2}-s)(u_n - a_N + 2)} \frac{(1-p^{n-2s})(1-\chi_N(p) p^{s-\frac{n}{2} -1})}{1-\chi_N(p)p^{\frac{n}{2}-s}},\] 
and if $n$ is odd we put
\[ H_n(N,s) = -\zeta_p(N) p^{\frac{1}{2} d_p(N) + (\frac{n}{2}-s)(u_n - a_{N_0} + 2)} \frac{(1-p^{n+1-2s})(1-\chi_{N_0}(p) p^{s-\frac{n-1}{2}})}{1-\chi_{N_0}(p)p^{\frac{n+1}{2}-s}}. \]
Then for $0 \le \nu \le n$, we have the inductive relation
\[ S_n^{(\nu)}(\chi_0,N',s) - S_n^{(\nu)}(\chi_0,N,s) = H_n(N,s) S^{(\nu)}_{n-1}(\chi_0,N_0,s). \]
\end{theorem}
We understand $S^{(n)}_{n-1}(\chi_0,N_0,s) = 0$, indeed if $\nu = n$ then $S_n^{(n)}(\chi_0,N',s) = S_n^{(n)}(\chi_0,N,s) = 1$.

For the proof we need the following.
\begin{lemma} \label{lem_EV_ramified_Siegel}
Let $r_n(s,\nu) = \dfrac{n(n+1)}{2}-\dfrac{\nu(\nu+1)}{2} -(n-\nu)s$.
Then for $\psi = \chi_0$ or $\chi_p$ we have
\[ S_n^{(\nu)}(\psi,pN,s) = p^{r_n(s,\nu)} S_n^{(\nu)}(\psi,N,s). \]
\end{lemma}

\begin{proof} By Proposition \ref{prop_Fourier_exp_Eisenstein_level_p_1}, the Fourier coefficient of $E^{n,(\nu)}_{k,\psi}(Z,s)$ at $N$ is given by
\begin{align*} A_\psi^{(\nu)}(N;Y,s) & = \det(Y)^{(s-k)/2} \Xi_n (Y,N, s_1,s_2 ) \\
& \qquad \times  S_n^{(\nu)}(\psi, N,p^{-s}) \prod_{q \ne p} S_{n,q}(N, \psi(q)q^{-s}), 
\end{align*}
with $s_1 = (s+k)/2$ and $s_2 = (s-k)/2$.
We have $A_\psi^{(\nu)}(pN;p^{-1}Y,s) = p^{l_n(s,\nu)} A_\psi^{(\nu)}(N;Y,s)$ with $l_n(s,\nu)$ defined in  (\ref{eq_l(s,nu)}).  On the other hand,  we have $\Xi_n(p^{-1}Y,pN, s_1,s_2) = p^{ns-n(n+1)/2} \Xi_n(Y,N,s_1,s_2)$ by \cite[(4.7.K), (4.34.K)]{Shi1}. Comparing the multiplication of $p$-power for both sides, we get the lemma.
\end{proof}

\begin{proof}[Proof of Theorem $\ref{thm_inductive_relation_chi_0}$]
Let $S_{n}(N,s)$ be the ordinary Siegel series at $p$. 
By \cite[Theorem 2.6]{Kat}, we have the recursion formula
\[
S_{n}(N',s) -p^{n+1-2s}S_{n}(N,s) = (1-p^{-s})(1+p^{1-s}) S_{n-1}(N_0,s-1). 
\]
Change $s \mapsto n+1-s$ and apply the functional equation in Theorem \ref{thm_FE_unram_Siegel_series} we can get
\begin{equation} S_n(N',s) - S_n(N,s) = H_n(N,s) S_{n-1}(N_0,s) \label{eq_ordinary_ind-rel}
\end{equation}
with $H_n(N,s)$ as above. Now we shall show that the same inductive relation holds for each characteristic ramified Siegel series.

Ordinary Siegel series $S_n(N,s)$ is the Euler $p$-factor of the Fourier coefficient of 
\[ E^n_k(Z,s) = \det(Y)^{(s-k)/2} \sum_{\g \in P_n(\Z) \backslash \Sp(n,\Z)} j(\g,Z)^{-k} |j(\g,Z)|^{-k+s}. \]
Since we have $E_k^n(Z,s) = \sum_{\nu=0}^n E^{n}_{k,\chi_0}(w_\nu;Z,s)$,  one can write
\[ E^k_n(Z,s) = \sum_{\nu=0}^n a_\nu(s) E^{n,(\nu)}_{k,\chi_0}(Z,s) 
\]
for some $a_\nu(s)$, that is a rational function in $p^{-s}$. 
Moreover since 
\[ \langle U(p^i) E^n_k(Z,s) \mid 0 \le i \le n \rangle_\C =  \mathcal{E}_{k,s}(\G_0^n(p),\chi_0) \]
by \cite[Section 4, Proposition]{Boe}, we have $a_\nu(s) \ne 0$ for each $\nu$. Thus together with (\ref{eq_ordinary_ind-rel}) we have
\[ \sum_{\nu=0}^{n-1} a_\nu(s) \bigl( S_{n}^{(\nu)}(\chi_0,N',s)-S_n^{(\nu)}(\chi_0,N,s) \bigr) = \sum_{\nu=0}^{n-1} a_\nu(s) H_n(N,s) S_{n-1}^{(\nu)} (\chi_0,N_0,s). \]
We regard both sides as functions in $N$. Lemma \ref{lem_EV_ramified_Siegel} says that $S_n^{(\nu)}(\chi_0,N,s)$ is an eigenfunction with eigenvalue $p^{r_n(s,\nu)}$, corresponding to the action $N \mapsto pN$. Since $H_n(pN,s) = p^{n-s} H_n(N,s)$ and $r_n(s,\nu) = r_{n-1}(s,\nu) + n-s$, take the same eigenvalue components from the both sides we have 
\[ S_n^{(\nu)}(\chi_0,N',s) - S_n^{(\nu)}(\chi_0,N,s) = H_n(N,s) S_{n-1}^{(\nu)}(\chi_0,N,s), \]
that is our assertion.
\end{proof}

Now we shall show the recursion formula of characteristic ramified Siegel series following \cite{Kat}. In Theorem \ref{thm_inductive_relation_chi_0} change $s \mapsto n+1-s$ and $\nu \mapsto n-\nu$ we have
\begin{align*}
&  S^{(n-\nu)}_n(\chi_0,N',n+1-s) - S^{(n-\nu)}_n(\chi_0,N,n+1-s) \\
 & \qquad   = H_n(N,n+1-s) S^{(n-\nu)}_{n-1}(\chi_0,N_0,n+1-s).
\end{align*}
We apply the functional equation in Theorem \ref{thm_FE_ramified_Siegel_series} to above. 
Note that $\beta^{(\nu)}_n(\chi_0,N',s) = \beta^{(\nu)}_n(\chi_0,N,s)$, $\zeta_p(N) = \zeta_p(N')$ and $e_p(N) = e_p(N')$. Hence we have
\begin{align*}
&  \zeta_p(N)p^{(s-\frac{n+1}{2})(e_p(N) +2) } \bigl(S_n^{(\nu)}(\chi_0,N',s)- p^{n+1-2s}S_n^{(\nu)}(\chi_0,N,s) \bigr) \\
& \quad  = \frac{\beta_n^{(\nu)}(\chi_0,N,s) \, \beta^{(n-\nu)}_{n-1}(\chi_0,N_0,n+1-s)}{\beta_n^{(n-\nu)}(\chi_0,N,n+1-s) \, \beta_{n-1}^{(\nu-1)}(\chi_0,N_0,s-1)} H_n(N,n+1-s) \\
& \quad \quad \times   \zeta_p(N_0)p^{(s-1-\frac{n}{2})e_p(N_0)} S^{(\nu-1)}_{n-1}(\chi_p,N_0,s-1).
\end{align*}
By a direct computation, we can rewrite it to the following simple relation.
\begin{equation*} 
S^{(\nu)}_n(\chi_0,N',s)-p^{n+1-2s} S^{(\nu)}_n(\chi_0,N,s) = (1-p^{\nu+1-2s}) S^{(\nu-1)}_{n-1}(\chi_0,N_0,s-1).
\end{equation*}
Together with Theorem \ref{thm_inductive_relation_chi_0} we have proved our main theorem.

\begin{theorem}[The recursion formula I] \label{thm_Main_theorem_1}
Let $0 \le u_1 \le \cdots \le u_n$ be integers and $\alpha_i \in \Z_p^\times$. Put $N_0 = \diag(\alpha_1 p^{u_1}, \ldots, \alpha_{n-1}p^{u_{n-1}}) \in \Sym^{n-1}(\Z_p)^*$
and $N = N_0 \perp \langle \alpha_n p^{e_n} \rangle$. Then we have the recursion formula
\[ S_n^{(\nu)}(\chi_0,N,s) = \frac{1-p^{\nu+1-2s}}{1-p^{n+1-2s}} S_{n-1}^{(\nu-1)}(\chi_0,N_0,s-1) - \frac{H_n(N,s)}{1-p^{n+1-2s}} S_{n-1}^{(\nu)}(\chi_0,N_0,s) . \]
Here $H_n(N,s)$ is given in Theorem \emph{\ref{thm_inductive_relation_chi_0}}.
\end{theorem}

\subsection{The case of quadratic character}

In the case of quadratic character, we cannot get the inductive relation from ordinary Siegel series. However Watanabe proved the following theorem in \cite{Wat2}.

\begin{theorem}[{\cite[Theorem 6.1]{Wat2}}] \label{thm_inductive_relation_ramified}
For any $\nu$ $(0 \le \nu \le n)$, we have
\[ S^{w_\nu}_n(\chi_p,N',s)-S^{w_\nu}_n(\chi_p,N,s) = H_n(\chi_p,N,s) S^{w_{\nu}}_{n-1}(\chi_p, N_0,s) \]
with a certain function $H_n(\chi_p,N,s)$, that is independent of $\nu$. Here we understand $S^{w_n}_{n-1}(\chi_p,N,s) = 0$.
\end{theorem}

\begin{remark} Watanabe gave the function $H_n(\chi_p,N,s)$ explicitly by dividing several cases. Here we shall give another description of $H_n(\chi_p,N,s)$ in Lemma \ref{lem_explicit_H(chi_p,N,s)} below. In order to show the theorem, Watanabe used his closed formula of $S^{w_\nu}_n(\chi_p,N,s)$, that we call ``Sato-Hironaka-type formula'' in Introduction.
\end{remark}

The crucial point is that $H_n(\chi_p,N,s)$ does not depend on $\nu$. From that we have for all $\nu$
\begin{equation} \label{inductive_relation_eigen-ramified}
 S^{(\nu)}_n(\chi_p,N',s) - S^{(\nu)}_n(\chi_p,N,s) = H_n(\chi_p,N,s) S^{(\nu)}_{n-1}(\chi_p,N_0,s),
\end{equation}
because in (\ref{ramified_Hecke_eigen_standard_base}) the coefficients $b_{\chi_p}(s)_{\nu r}$ are common for $S^{(\nu)}_n$ and $S^{(\nu)}_{n-1}$.
One can find the following explicit form of $H_n(\chi_p,N,s)$ by applying this relation to $\nu=0$, since we know the explicit formula of $S^{(0)}_n(\chi_p,N,s)$ by Corollary \ref{cor_explicit_formula_S_n^(n)}.

\begin{lemma} \label{lem_explicit_H(chi_p,N,s)}
Let $N = \diag(\alpha_1 p^{u_1}, \cdots, \alpha_n p^{u_n})$.
\begin{enumerate}[$(1)$]
\item If $n$ is even then 
\begin{align*} H_n(\chi_p,N,s) & =  \ve_p \eta_p(N_0) p^{\frac{1}{2} d_p(N_0) + (\frac{n+1}{2}-s)(u_n + a_N+1)}  \\
& \qquad \times \frac{ (1-p^{n-2s})(1-\chi_N^*(p)p^{s-\frac{n}{2}-1})}{1-\chi_N^*(p)p^{\frac{n}{2}-s}}. 
\end{align*}
\item If $n$ is odd then
\begin{align*}
H_n(\chi_p,N,s) & = -\ve_p \eta_p(N) p^{\frac{1}{2}d_p(N) + (\frac{n}{2}-s)(u_n + a_{N_0}+1 )} \\
& \qquad \times  \frac{ (1-p^{n+1-2s})(1-\chi_{N_0}^*(p)p^{s-\frac{n-1}{2}})}{1-\chi_{N_0}^*(p)p^{\frac{n+1}{2}-s}}. 
\end{align*}
\end{enumerate}
\end{lemma}

By a similar argument to the case of trivial character, we have the following result.

\begin{theorem}[The recursion formula II] \label{thm_Main_theorem_2}
Let $0 \le u_1 \le \cdots \le u_n$ be integers and $\alpha_i \in \Z_p^\times$. Put $N_0 = \diag(\alpha_1 p^{u_1}, \ldots, \alpha_{n-1}p^{u_{n-1}}) \in \Sym^{n-1}(\Z_p)^*$
and $N = N_0 \perp \langle \alpha_n p^{u_n} \rangle$. Then we have the recursion formula
\[ S_n^{(\nu)}(\chi_p,N,s) = \frac{1-p^{\nu+1-2s}}{1-p^{n+1-2s}} S_{n-1}^{(\nu-1)}(\chi_p,N_0,s-1) - \frac{H_n(\chi_p,N,s)}{1-p^{n+1-2s}} S_{n-1}^{(\nu)}(\chi_p,N_0,s) . \]
Here $H_n(\chi_p,N,s)$ is given in Lemma \emph{\ref{lem_explicit_H(chi_p,N,s)}}.
\end{theorem}

%
%

As a consequence, we can easily compute the Fourier coefficients of Siegel Eisenstein series $E^{n}_{k,\psi}(w_\nu;Z)$, by Theorem \ref{thm_Main_theorem_1} and Theorem \ref{thm_Main_theorem_2} together with (\ref{ramified_standard_base_Hecke_eigen}) and Proposition \ref{prop_explicit_c_psi_ij}.

\section{Examples} \label{Sec_example}

We compute the ramified Siegel series for low degree case by using our recursion formulas. We treat the case of quadratic character: $\psi = \chi_p$.

The degree $1$ case, that is a starting point, the following formula is well-known. For $N=  \alpha p^m$ with $\alpha \in \Z_p^\times$, 
\begin{align*}
S^{w_0}_1(\chi_p, N, s) = S^{(0)}_1(\chi_p, N, s) & = \chi_p(\alpha) \ve_p p^{(1-s)m + 1/2-s} \\
S_1^{w_1}(\chi_p, N,s) = S^{(1)}_1(\chi_p, N, s) & = 1.
\end{align*}
The former equation also follows from Corollary \ref{cor_explicit_formula_S_n^(n)}.

In degree 2 case, we compute $S^{w_1}_2(\chi_p,N,s) = S^{(1)}_2(\chi_p,N,s)$. 
Let $N_0 = \alpha p^m$ and $N = p^m \diag(\alpha , \beta p^r)$ with $\alpha, \beta \in \Z_p^\times$. We put $l_N = [\frac{r+1}{2}]$, then in $H_2(\chi_p,N,s)$ the term $u_2+a_N = m+2l_N$. Theorem \ref{thm_Main_theorem_2} says
\begin{align*} S^{(1)}_2(\chi_p,N,s) & = \frac{1-p^{2-2s}}{1-p^{3-2s}}S_1^{(0)}(\chi_p,N_0,s-1) - \frac{H_2(\chi_p,N,s)}{1-p^{3-2s}} \\
& = \frac{\chi_p(\alpha) p^{(2-s)m+3/2-s}(1-p^{2-2s})}{(1-p^{3-2s})(1-\chi_N^*(p)p^{1-s})}  \\
& \qquad  \quad \times \bigl( 1-p^{(3-2s)l_N} -\chi_N^*(p)p^{1-s}(1-p^{(3-2s)(l_N-1)}) \bigr),
\end{align*}
that coincides with \cite[Proposition 3.2]{Gu2}.

Next we consider the case of degree 3. The author has already calculated $S^{w_0}_3(\chi_p,N,s)$ in \cite{Gu3}, we shall re-compute it by our recursion formula. Let $N_0 = p^m(\alpha, p^r \beta)$, $N = N_0 \perp \langle \g p^{m+r+t}  \rangle$ and $l_{N_0} = [ \frac{r+1}{2} ]$ as above, moreover $f_{N_0} = m + [\frac{r}{2}]$.
By (\ref{ramified_standard_base_Hecke_eigen}), Proposition \ref{prop_explicit_c_psi_ij} and Corollary \ref{cor_explicit_formula_S_n^(n)} we have
\begin{align*} S_3^{w_0}(\chi_p, N,s ) & = -\frac{\chi_p(-1) \eta_p(N) \ve_p p^{(2-s)d_p(N)+7/2-3s}(1-p^{2-2s})}{(1-p^{3-2s})} \\
& + \frac{\chi_p(-1)p^{1-2s}(p-1)}{(1-p^{3-2s})} S_3^{(2)}(\chi_p,N,s).
\end{align*}
By Theorem \ref{thm_Main_theorem_2},
\begin{align*} & \quad S_3^{(2)}(\chi_p,N,s)  = \frac{1}{1-p^{4-2s}}( (1-p^{3-2s})S_2^{(1)}(\chi_p,N,s-1) - H_3(\chi_p,N,s)) \\
& = \frac{\eta_p(N) \ve_p p^{(2-s)d_p(N)-(3-2s)f_{N_0} + 3/2-s} (1-\chi_{N_0}^*(p)p^{s-1})}{1-\chi_{N_0}^*(p) p^{2-s}} \\
& + \frac{\chi_p(\alpha) p^{(3-s)m+5/2-s}(1-p^{3-2s})}{(1-p^{5-2s})(1-\chi_{N_0}^*(p)p^{2-s})}  \bigl( 1-p^{(5-2s)l_{N_0}} -\chi_{N_0}^*(p)p^{2-s}(1-p^{(5-2s)(l_{N_0}-1)}) \bigr).
\end{align*}
We can get the same result as \cite[Theorem 4.1]{Gu3}.

\end{document}